\documentclass[11pt]{amsart}
\usepackage[english]{babel}
\usepackage{egothic}
\usepackage[T1]{fontenc}

\usepackage{amsmath}
\usepackage{amsthm}
\usepackage{amssymb}
\usepackage[all]{xy}
\usepackage{mathrsfs}
\usepackage{todonotes}
\usepackage{bbm}
\usepackage{enumerate}
\usepackage{float}
\usepackage{graphicx}
\usepackage{caption}
\usepackage{subcaption}
\usepackage{xcolor} 
\usepackage{hyperref}

\usepackage{accents}

\newtheorem{theorem}{Theorem}[section]

\newtheorem{lemma}[theorem]{Lemma}
\newtheorem{proposition}[theorem]{Proposition}
\newtheorem{corollary}[theorem]{Corollary}

\theoremstyle{definition}
\newtheorem{remark}[theorem]{Remark}

\newcommand{\E}{\mathbb{E}}

\renewcommand{\1}{{\mathbbm{1}}}

\DeclareRobustCommand{\ctau}{{
		\mathpalette\cap@greek\tau
}}
\DeclareRobustCommand{\csigma}{{
		\mathpalette\cap@greek\sigma
}}
\makeatother

\title[Bounds for Dyadic Square Functions of indicator functions]{{Sharp} Lower Bounds for Dyadic Square Functions of indicator functions of sets}
\date{}

\author[N. Alpay]{Natanael Alpay}
\address{(NA)
Department of Mathematics\\ 
University of California, Irvine,
Irvine, CA 92697 \\
USA}
\email{nalpay@uci.edu}

\author[P. Ivanisvili]{Paata Ivanisvili}
\address{(PI)
Department of Mathematics\\ 
University of California, Irvine,
Irvine, CA 92697 \\
USA}
\email{pivanisv@uci.edu}


\begin{document}

\begin{abstract}
We study lower bounds for dyadic square functions of indicator functions. In the case of the dyadic square function $S_{2}$ we obtain a sharp lower bound: for every measurable $A \subset {[0,1)}$, we have
\[
\|S_{2}(\mathbbm{1}_{A})\|_{1}\ge \mathbb{E}_{|A|}\big[\sqrt{\tau}\big]\asymp |A|^{*}\log_2\frac{1}{|A|^{*}},
\]
where $\tau$ is the first exit time from $(0,1)$ of a standard Brownian motion started at $|A|$, and $|A|^{*}:=\min\{|A|,1-|A|\}$.  
This estimate gives logarithmic improvement over the classical Burkholder--Davis--Gundy lower bound $|A|^{*}$. 
In addition, we show a sharp inequality
\[
\|S_{1}(\mathbbm{1}_{A})\|_{1} \ge T(|A|)\asymp  |A|^{*}\log_{2}\frac{1}{|A|^{*}},
\]
where $T(x)=\sum_{k=0}^{\infty}\frac{\operatorname{dist}(2^{k}x,\mathbb{Z})}{2^{k}}$ is the Takagi function. 
\end{abstract}

\maketitle

\section{Introduction and the main  results}

Let  $({[0,1)}, \mathcal{B}, dx)$ be the probability space, where $\mathcal{B}$ is the Borel $\sigma$-algebra, and $dx$ is the Lebesgue measure. For each $n \geq 0$ we denote by $\mathcal{D}_n$ dyadic intervals belonging to {$[0,1)$} of level $n$, i.e.,
\[
\mathcal{D}_n=
{\left\{
\left[\frac{k}{2^n}, \frac{k+1}{2^n}\right)
:\ k=0,\ldots,2^n-1
\right\}.}
\]

Let $\{{[0,1)}, \emptyset\} =\mathcal{F}_{0} \subset \mathcal{F}_{1} \subset \ldots$ be the sequence of increasing family of filtrations, i.e., $\sigma$-algebras generated by the dyadic intervals $\mathcal{D}_{n}$, and let $\mathcal{D} = \bigcup_{n\geq 0} \mathcal{F}_n$ denote the family of all sets measurable with respect to some finite dyadic level. Let us also introduce a symbol $D_{n}$ denoting dyadic numbers of level $n$, i.e., $D_{n} = \{ k/2^{n}, k=0, \ldots, 2^{n}\}$, and set $D= \cup_{n \geq 0} D_{n}$. For any $f \in {L^{1}([0,1))}$, and any interval $I \subset {[0,1)}$ we set
\begin{align*}
\langle f \rangle_{I} = \frac{1}{|I|} \int_{I} f(y)dy,
\end{align*}
where $|I|$ denotes the Lebesgue length of the interval. Given $f \in {L^{1}([0,1))}$, the sequence $\{f_{n}\}_{n \geq 0}$, where
\begin{align*}
f_{n} = \mathbb{E} (f|\mathcal{F}_{n}) = \sum_{I\in \mathcal{D}_n} \mathbbm{1}_{I}(x)\langle f \rangle_{I},
\end{align*}
is called the dyadic martingale generated by $f$. Clearly $\lim_{n \to \infty} f_{n} =f$ a.e. We say that $\{f_{n}\}$ is a simple martingale (also called Paley--Walsh martingale \cite{ABS}) if $f_{N}=f_{N+1}=\ldots$ after some large $N$. For general Borel sets $A\subset {[0,1)}$, the martingale generated by $\mathbbm{1}_{A}$ need not be simple; simple martingales will nevertheless play an important role as approximants and in the sharpness constructions. Define $\{d_{n}\}_{n \geq 1}$, where $d_{n} := f_{n}-f_{n-1}$, to be the martingale difference sequence. Next, for $\beta \geq 1$ and $N\ge1$, we define
\begin{align*}
S_{\beta,N}(f)=\left(\sum_{1\leq n \leq N} |d_{n}|^{\beta}\right)^{1/\beta}.
\end{align*}
We then define $\beta$-variation $S_{\beta}(f)$ by
\begin{align*}
S_{\beta}(f) = \lim_{N \to \infty} S_{\beta,N}(f)
= \left(\sum_{n \geq 1} |d_{n}|^{\beta}\right)^{1/\beta} \; \in \;  [0, \infty]. 
\end{align*}
When $\beta=2$ the function $S_{2}$ is also known as the (dyadic) square function. 


The celebrated Burkholder--Davis--Gundy inequality states that the $L_{p}$ norm of the square function $S_{2}(f)$ is comparable to the $L_{p}$ norm of $f$ for all $p$, $1<p<\infty$ (see \cite{BDS1}, \cite{BDS2}). For dyadic martingales, we have better bounds (see \cite{Wang1}, \cite{Wang2}, \cite{Davis} and references therein): there exist universal constants $0<c,C<\infty$ such that
\begin{align}\label{BDG}
c\frac{\sqrt{p}}{p+1}\| f-\mathbb{E}f\|_{p} \stackrel{0<p<\infty}{\leq} \|S_{2}(f)\|_{p} \stackrel{1<p<\infty}{\leq} C \frac{p^{3/2}}{p-1} \| f -\mathbb{E}f\|_{p}
\end{align}
holds for all  $f \in {L^{p}([0,1))}$ (here $\mathbb{E}f = \int_{0}^{1}f$). The endpoint behavior of the constants in (\ref{BDG}) are sharp. \\

In this paper we will be mostly concerned with lower bounds on $\|S_{2}(f)\|_{1}$ and $\|S_{1}(f)\|_{1}$ in the case when $f$ is a boolean function, i.e., $f=\mathbbm{1}_{A}$ for some Borel $A\subset  {[0,1)}$. Applying (\ref{BDG}) to $f(x)= \mathbbm{1}_{A}(x)$ we obtain 
\begin{align}\label{bestBDG}
\|S_{2}(\mathbbm{1}_{A})\|_{1} \gtrsim |A|^{*},
\end{align}
where $f \gtrsim g$ means $f \geq C g$ for some universal constant $C>0$, and  
$$
|A|^{*} := \min\{ |A|, 1-|A|\} \asymp |A| (1-|A|).
$$
Here the symbol $f\asymp g$ means that there exist two universal constants $c, C>0$ such that $c\leq \frac{f(x)}{g(x)} \leq C$. Our first main result is the following.
\begin{theorem}\label{mth01}
Let $(B_t)_{t\ge 0}$ be a standard Brownian motion started at $B_0=p\in[0,1]$, and let
\[
\tau=\inf\{t\ge 0:\, B_t\notin (0,1)\}
\]
be its first exit time from $(0,1)$. We have
\[
\inf_{A}\|S_{2}(\mathbbm{1}_{A})\|_{1}=\mathbb{E}_{p}\sqrt{\tau}.
\]
Here, infimum is taken over all Borel $A \subset {[0,1)}$ of measure $|A|=p$. 
\end{theorem}

\begin{proposition}\label{prop:const}
    For all $p \in [0,1]$ we have 
\[
\sqrt{\frac2\pi}\, \; p^*\log\frac1{p^*}
\;\le\;
\mathbb E_p\sqrt{\tau}
\;\le\;
\frac{4\sqrt2\,G}{\pi^{3/2}\ln 2}\,p^*\log\frac1{p^*},
\]
where $p^*=\min\{p,1-p\}$, and both constants $\sqrt{\frac2\pi}$ and $\frac{4\sqrt2\,G}{\pi^{3/2}\ln 2}$   are sharp.
\end{proposition}

For comparison, our result improves the best known bound (\ref{bestBDG}) by a factor of $\log(1/|A|^*)$. 

\vskip0.5cm 

Fix some parameters $(\alpha, \beta)$ such that $\beta \geq 1 \geq \alpha >0$. Define $\mathcal{F}_{\alpha,\beta}$ to be the family of all non-negative continuous functions $f$ defined on $[0,1]$, satisfying the following {\em two-point inequality} 
\begin{align}\label{twopp}
f^{\alpha}\left(\frac{x+y}{2}\right) \leq \frac{1}{2} \left(f^{\beta}(x)+\left|\frac{x-y}{2}\right|^{\beta}\right)^{\frac{\alpha}{\beta}}+\frac{1}{2} \left(f^{\beta}(y)+\left|\frac{x-y}{2}\right|^{\beta}\right)^{\frac{\alpha}{\beta}}
\end{align}
for all $x,y \in [0,1]$, and having the boundary condition
$$
f(0)=f(1)=0.
$$

\begin{proposition}\label{twopp-max-func}
For all $\beta \geq 1 \geq \alpha >0$, there exists a pointwise maximal {continous} function $ B_{\alpha, \beta}$ belonging to $\mathcal{F}_{\alpha,\beta}$, satisfying
the {\em two-point inequality} \eqref{twopp}
for all $x,y \in [0,1]$, and having the boundary condition  $B_{\alpha, \beta}(0)=B_{\alpha, \beta}(1)=0$.
\end{proposition}

\begin{theorem} \label{mthgenab}
For all pairs $(\alpha, \beta)$ with $\beta \geq 1\geq \alpha >0$, we have
\begin{align}\label{genab}
\| S_{\beta}(\mathbbm{1}_{A})\|_{\alpha} \geq B_{\alpha,\beta}(|A|). 
\end{align}
for all Borel $A \subset {[0,1)}$.
\end{theorem}

\begin{remark}\label{remb1}[Gaussian isoperimetric profile $B_{1,2}$]
By a theorem of Bobkov we have $B_{1,2}(x)=I(x)$ (see \cite{bobk1}). 
And so a direct application of Theorem~\ref{mthgenab} with $(\alpha,\beta)=(1,2)$ and Bobkov's extremal function $B_{1,2}=I$ yields
\begin{equation}
\label{eq:suboptbound}
    \|S_{2}(\mathbbm{1}_{A})\|_{1} \geq I(|A|),
\end{equation}
where $I(x)$ is the Gaussian isoperimetric profile
\begin{align*}
I(x)=\Phi'(\Phi^{-1}(x)), \quad \text{where} \quad \Phi(x)=\int_{-\infty}^{x}\frac{e^{-t^{2}/2}}{\sqrt{2\pi}}dt.
\end{align*}
It is known (see also \cite{bobk1}) that 
\[
I(x) \asymp x^{*} \sqrt{\log_{2}(1/x^{*})}.
\]
The bound \eqref{eq:suboptbound} is \emph{suboptimal} compared to Theorem~\ref{mth01} (it has only a square root on the logarithm). The proof of Theorem~\ref{mth01} uses a different Bellman function, defined in terms of the Brownian exit time~$\tau$, and is given in Section~\ref{sec:proof-mth01}.
\end{remark}

\begin{remark}\label{supersol}
The conclusion of Theorem~\ref{mthgenab} holds with any continuous $\widetilde{B}_{\alpha, \beta}$ satisfying condition $\widetilde{B}_{\alpha, \beta}(0)=\widetilde{B}_{\alpha, \beta}(1)=0$, and two-point inequality~\eqref{twopp}. 
\end{remark}

\begin{remark}[Comparison with an earlier version]
An earlier version of this manuscript obtained only the Gaussian-profile lower bound
\eqref{eq:suboptbound}. Theorem~\ref{mth01} is strictly stronger: for each
$p\in[0,1]$,
\[
\inf_{|A|=p}\|S_{2}(\mathbbm{1}_{A})\|_{1}
=
\mathbb E_p\sqrt{\tau},
\]
so the Brownian exit-time profile $\mathbb E_p\sqrt{\tau}$ is the exact sharp
lower profile for the $S_2$ problem, rather than merely a Gaussian-type lower
bound.
\end{remark}

It turns out that Theorem~\ref{mthgenab} is sharp for the pair $(\alpha, \beta)=(1, 1)$.

\begin{theorem}\label{mth02}
We have
\begin{align*}
\inf_{A} \| S_{1}(\mathbbm{1}_{A})\|_{1} = B_{1,1}(x)
\end{align*}
for all $x\in [0,1]$, where the infimum is taken over all Borel $A\subset{[0,1)}$ of measure $x$.
\end{theorem}

The function $B_{1,1}(x)$, unlike $B_{1,2}(x)=I(x)$, is not differentiable on $(0,1)$.
We will see that $B_{1,1}$ has a fractal-like structure, i.e., it will satisfy the functional equation
$$
B_{1,1}(x)+x=2B_{1,1}(x/2).
$$
We will also see that 
\begin{align*}
B_{1,1}(x) \geq x^{*} \log_{2}\frac{1}{x^{*}} \quad \text{for all} \quad x \in [0,1]
\end{align*}
with equality whenever $x=2^{-k}$ or $x=1-2^{-k}$ for any nonnegative integer $k\geq 0$, see Figure \ref{fig1}.\\

\begin{figure}
    \centering
    \IfFileExists{Bn4.png}{\includegraphics[scale=0.5]{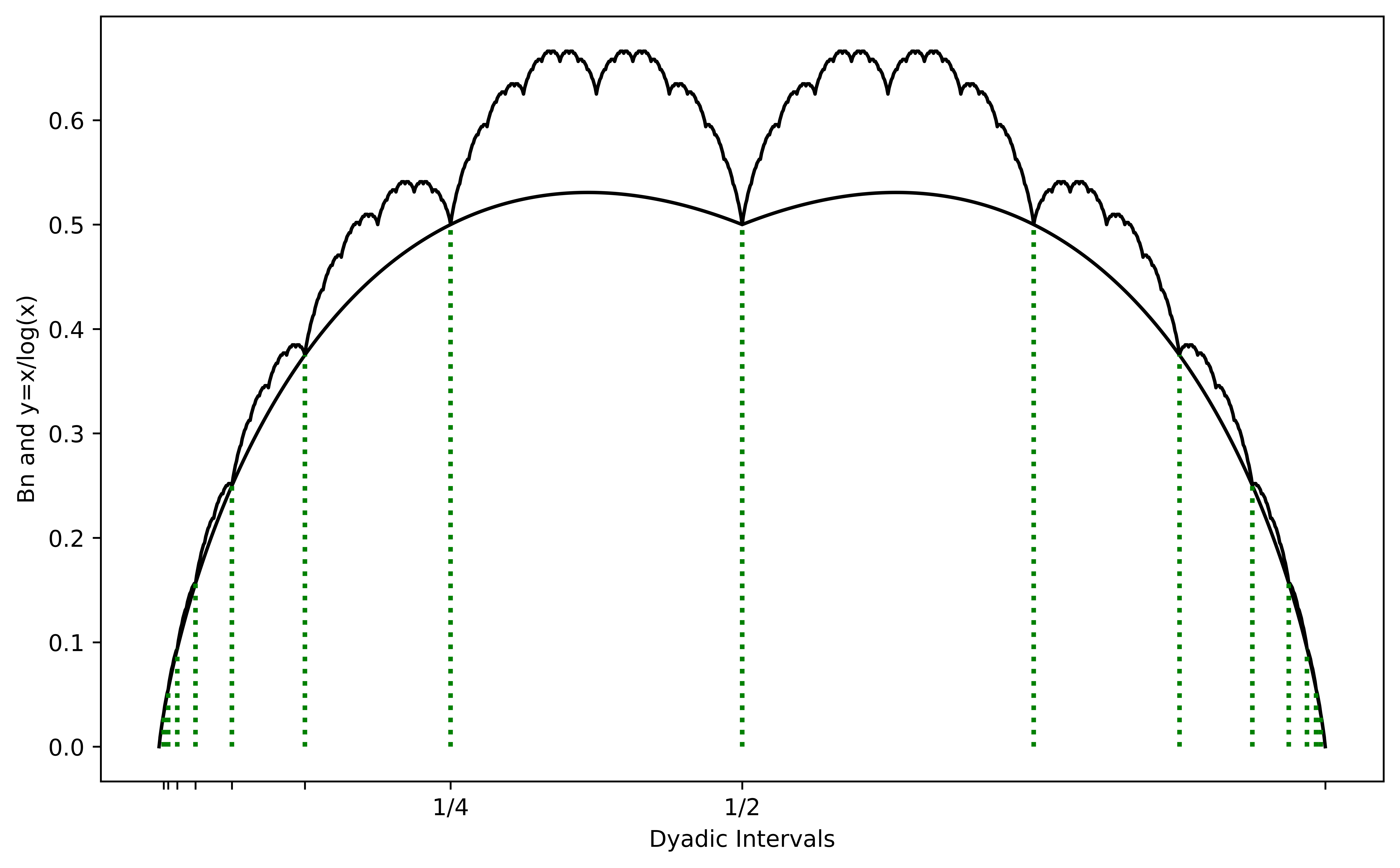}}{\fbox{\texttt{Bn4.png} missing}}
    \caption{$B_{1,1}$ and $x^*\log_2(1/x^*)$. }
    \label{fig1}
\end{figure}


For any integer $n \geq 0$, let $s(n)$ be the sum of 1's in the binary representation of $n$.
 For any integer $k\geq0$, set 
\begin{equation}\label{eq:F_def}
F(k) = \sum_{j=0}^{k-1}s(j)
\end{equation}
and let $F(0)=0$. Define the sequence of functions $B_{n} : D_{n} \mapsto [0,\infty)$ as follows 
\begin{align}
B_n(x)    =   nx - \frac{1}{2^{n-1}} F(x2^n),\label{eq:bn}
\end{align}
where we recall $D_{n} = \{ \frac{k}{2^{n}}, k=0, \ldots, 2^{n}\}$. We will see that $B_{n+1}|_{D_{n}}=B_{n}$ for all $n\geq 1$.

\begin{theorem}\label{thm-bn}
We have 
\begin{itemize}
\item[(i)]
For any $x \in D$, the limit $\lim_{n \to \infty}B_{n}(x)$ exists and is denoted by $P(x)$. 

\item[(ii)] The limit $P: D \mapsto [0,\infty)$ is the pointwise maximal function defined on $D$ satisfying $P(0)=P(1)=0$, and the two-point inequality \eqref{twopp} with $(\alpha, \beta)=(1,1)$. 

\item[(iii)] There exists a universal constant $C>0$ such that
\begin{align*}
|P(x)-P(y)| \leq C |x-y| \log\left(\frac{1}{|x-y|}\right)
\end{align*}
for all $x,y \in D$ with $|x-y|<\frac{1}{2}$.
\item[(iv)] For all $x \in D$ we have $P(x)=P(1-x)$ and $P(x)+x = 2P(x/2)$. 
\item[(v)] For all $x \in D$ we have 
\begin{align*}
P(x) \geq x^{*} \log_{2}(1/x^{*})  
\end{align*}
with equality\footnote{Recall that $x^{*}=\min\{x, 1-x\}$} at points $x=2^{-k}$ and $x=1-2^{-k}$ for all nonnegative integers $k$.
\item[(vi)] We have $B_{1,1}|_{D}=P$, and $B_{1,1}$ satisfies \textup{(iii)}--\textup{(v)} for all $x \in [0,1]$. 
\end{itemize}
\end{theorem}

After we wrote the first version of this paper, we learned from Grok that the function $B_{1,1}$ (equivalently, the continuous extension of the  limit $P$ from Theorem~\ref{thm-bn}) coincides with the \emph{Takagi function} \cite{AllaartKawamuraTakagiSurvey,LagariasTakagiRH}, also known as the \emph{blancmange curve}, given by 
\begin{equation}\label{eq:Tagaki}
 T(x)=\sum_{k=0}^{\infty}\frac{\operatorname{dist}(2^{k}x,\mathbb{Z})}{2^{k}}. 
\end{equation}
Namely, the class of functions $\mathcal{F}_{2}$ studied in  \cite{Lev} is equivalent (up to constant multiple of $2$) to $\mathcal{F}_{1,1}$ in our paper, and so by {\cite[Theorem 7]{Lev}} we get the following 
$$ B_{1,1}(x)= T(x)
.$$

\begin{proposition}\label{col:asympT}
Let $x^*=\min(x,1-x)$. Then
$$
T(x)\asymp x^*\log_2\frac1{x^*} \quad \text{on} \quad [0,1].
$$
\end{proposition}

We decided to move our proof of Theorem~\ref{thm-bn} to Appendix but still keep it in this paper because our proof is independent from the one in the literature and, moreover,  it gives one more description of  Takagi function $T(x)$ as $\inf_{A \subset {[0,1)},\, |A|=x} \| S_{1}(\mathbbm{1}_{A})\|_{1}$
for all $x \in [0,1]$, where the infimum is taken over all Borel sets $A\subset{[0,1)}$ of measure $x$. \\

\begin{remark}\label{rem:takagi-rh}
The Takagi function also appears in analytic number theory. In particular, work of Kanemitsu--Yoshimoto shows that the Riemann hypothesis is equivalent to a discrepancy estimate for the averages of the Takagi function along Farey fractions; see \cite{KYTakagiRH,LagariasTakagiRH} for details.
\end{remark}

In the next theorem, for each fixed $\alpha \in (0,1)$ we obtain lower bounds on $\| S_{1}(\mathbbm{1}_{A})\|_{\alpha}$ which are sharp (up to a constant factor). 

\begin{theorem}\label{mth03}
For each $\alpha \in (0,1)$ we have
\begin{align*}
\|S_{1}(\mathbbm{1}_{A})\|_{\alpha} \geq |A|^{*}
\end{align*}
holds for any Borel $A \subset {[0,1)}$. Moreover for each $\alpha\in (0,1)$ there exists a constant $C_{\alpha}>0$ and a sequence of sets $A_{j} \in \mathcal{D}$ with nonzero measure such that $\lim_{j \to \infty}|A_{j}|=0$ and
\[
\|S_{1}(\mathbbm{1}_{A_{j}})\|_{\alpha} \leq C_{\alpha} |A_{j}|.
\]
\end{theorem}

Comparing Theorems~\ref{mth02} and \ref{mth03} we see that the sharp lower bounds on $\| S_{1}(\mathbbm{1}_{A})\|_{\alpha}$ are $|A|\log(1/|A|)$ and $|A|$ (assuming $|A|\leq1/2$) when $\alpha =1$ and $\alpha \in (0,1)$ correspondingly. Surprisingly to us, there is no such threshold behaviour in the case of a discrete gradient on the hypercube, which we discuss in the next section.\\

\begin{remark}\label{rmk:Lebesgue}
Although we formulate our results for Borel sets, all statements extend verbatim to arbitrary Lebesgue measurable sets. Indeed, for every Lebesgue measurable set $A \subset [0,1)$ there exists a Borel set $B \subset [0,1)$ such that $|A \triangle B| = 0$. Since all dyadic conditional expectations and the associated square functions $S_\beta(\mathbbm{1}_A)$ depend only on the $L^1$-equivalence class of $\mathbbm{1}_A$, it follows that
\[
S_\beta(\mathbbm{1}_A) = S_\beta(\mathbbm{1}_B) \quad \text{a.e.}
\]
and hence
\[
\|S_\beta(\mathbbm{1}_A)\|_\alpha = \|S_\beta(\mathbbm{1}_B)\|_\alpha.
\]
In particular, any inequality or extremal problem stated for Borel sets remains valid for all Lebesgue measurable sets.
\end{remark}

The purpose of the next section, which also served as the original motivation for this project, is to explore the parallels and distinctions between lower bounds for discrete gradients of Boolean functions on the hypercube (the so-called \emph{edge isoperimetric inequalities}) and dyadic square functions of indicator functions of sets.

\subsection{Discrete gradient on the hypercube}\label{hypercc}
Another way to model dyadic martingales is through the hypercube. Let $n\geq 1$, and consider the $n$-dimensional hypercube $\{-1,1\}^{n}$ equipped with uniform probability measure. Let $x_{1}, \ldots, x_{n}$ be independent identically distributed Bernoulli symmetric  $\pm 1$ random variables. Set $x = (x_{1}, \ldots, x_{n}) \in \{-1,1\}^{n}$. For any $f:\{-1,1\}^{n} \mapsto \mathbb{R}$, the sequence
\begin{align*}
&f_{0}=\mathbb{E}f(x);\\
&f_{1}(\varepsilon_{1}) = \mathbb{E}(f(x)|x_{1}=\varepsilon_{1});\\
&f_{2}(\varepsilon_{1}, \varepsilon_{2}) = \mathbb{E}(f(x)|x_{1}=\varepsilon_{1}, x_{2}=\varepsilon_{2});\\
&\ldots \\
&f_{n}(\varepsilon_{1}, \ldots, \varepsilon_{n}) = \mathbb{E} (f(x)|x_{1}=\varepsilon_{1}, \ldots, x_{n}=\varepsilon_{n}) =f(\varepsilon_{1}, \ldots, \varepsilon_{n}). 
\end{align*}
defines a dyadic martingale $\{f_{k}\}_{k=0}^{n}$. Notice that here 
\begin{align*}
&f_{0} = \frac{1}{2^{n}}\sum_{x_{1}, \ldots, x_{n}=\pm 1} f(x_{1}, \ldots, x_{n}),\\
&f_{1}(\varepsilon_{1}) = \frac{1}{2^{n-1}}\sum_{x_{2}, \ldots, x_{n}=\pm 1} f(x_{1}, x_{2}, \ldots, x_{n}),\\
&\ldots\\
&f_{n-1}(\varepsilon_{1}, \ldots, \varepsilon_{n-1})=\frac{f(\varepsilon_{1}, \ldots, \varepsilon_{n-1}, 1)+f(\varepsilon_{1}, \ldots, \varepsilon_{n-1}, -1)}{2}.
\end{align*}
Next, we define discrete derivatives. For each $j \in \{1, \ldots, n\}$ we set 
\begin{align*}
    &D_{j}f(x) = 
    \frac{f(x_{1}, \ldots, x_{j-1}, x_{j}, x_{j+1}, \ldots, x_{n})-f(x_{1}, \ldots, x_{j-1}, -x_{j}, x_{j+1}, \ldots, x_{n})}{2}.
\end{align*}
Let $\nabla f(x)= (D_{1}f(x), \ldots, D_{n}f(x))$, and for each $\beta \geq 1$  set $\beta$- gradient to be defined as 
\begin{align*}
    |\nabla f|_{\beta}(x) = \left( \sum_{j=1}^{n} |D_{j}f(x)|^{\beta}\right)^{1/\beta}.
\end{align*}

In general the functions $S_{\beta}(f)$ and $|\nabla f|_{\beta}$ are not comparable. One may ask: how different are they? Let us make a couple of observations. 

\begin{itemize}
\item[1.]If $f$ is boolean, i.e., $f(x) = \mathbbm{1}_{A}$ for some $A\subset  \{-1,1\}^{n}$ then $|\nabla f|_{\beta}$ equals to $(|\nabla f|_{1})^{1/\beta}$ up to a constant factor. Indeed, since $|2D_{j}f| \in \{0,1\}$, we have
\begin{align*}
|\nabla f|_{\beta} = \frac{1}{2}\left( \sum_{j=1}^{n} |2D_{j}f|\right)^{1/\beta} = 2^{\frac{1}{\beta}-1} (|\nabla f|_{1})^{1/\beta}. 
\end{align*}
Therefore  for boolean $f$ we always have
\begin{align}\label{giden}
\|\, |\nabla f|_{\beta} \,\|_{p} = 2^{\frac{1}{\beta}-1}  \|\, |\nabla f|_{1}\,\|_{p/\beta}^{1/\beta}.
\end{align}

The identity (\ref{giden}) shows that studying lower bounds on  $\|\, |\nabla f|_{\beta} \,\|_{p}$ is the same as studying lower bounds on $\|\, |\nabla f|_{1} \,\|_{q}$ with $q=p/\beta$.

\item[2.] Talagrand's  isoperimetric inequality with ``two-sided gradient''    says that  there exists a universal constant $C>0$ such that for  any $A\subset  \{-1,1\}^{n}$ we have 
\begin{align}\label{tal001}
    \|  |\nabla \mathbbm{1}_{A}|_{1} \|_{q} \geq C  (|A|^{*})^{1/q} \log\frac{1}{|A|^{*}}
\end{align}
holds for all $q\geq 1/2$. The estimate (\ref{tal001}) is sharp up to a universal constant factor $C$. Moreover, as soon as $q<1/2$ there is no lower bound on $\|  |\nabla f|_{1} \|_{q}$ in terms of $|A|$ independent of $n$ (see \cite{BIM} and references therein). See \cite{DIR} about the state of the art of the best constant $C$ in (\ref{tal001}).

Notice that when $q=1$ the left-hand side in (\ref{tal001}) counts the number of edges joining $A$ and $A^{c}$, and it gives the lower bound of the edge boundary of $A$ in terms of the size of $A$. This inequality ($q=1$) is known as the edge-isoperimetric inequality on the Hamming cube, see \cite{hart}. 

Combining (\ref{tal001}) and (\ref{giden}) we get that for all $2p \geq \beta \geq 1$ and any $A\subset  \{-1,1\}^{n}$  we have 
\begin{align}\label{taliso}
    \|\, |\nabla \mathbbm{1}_{A}|_{\beta} \,\|_{p} \geq C  (|A|^{*})^{1/p} \left(\log\frac{1}{|A|^{*}}\right)^{1/\beta}.
\end{align}

The reader should compare (\ref{taliso}) with lower bounds for $\| S_{\beta} (\mathbbm{1}_{A}) \|_{p}$ obtained in Theorems~\ref{mth01}, \ref{mth02} and \ref{mth03}. 
\item[3.] We have 
\begin{align*}
    &(S_{\beta}f(x))^{\beta} = \sum_{k=1}^{n} |f_{k}-f_{k-1}|^{\beta}= |\mathbb{E} (D_{1}f(x)|x_{1})|^{\beta}\\
    &+|\mathbb{E} (D_{2}f(x)|x_{1}, x_{2})|^{\beta}+\ldots+|\mathbb{E} (D_{n}f(x)|x_{1}, x_{2}, \ldots, x_{n})|^{\beta}.
\end{align*}
Thus we obtain 
\begin{align*}
\| |\nabla \mathbbm{1}_{A}|_{1}\|_{1} \geq \| S_{1}(\mathbbm{1}_{A})\|_{1} \stackrel{\mathrm{Theorem~\ref{mth02}}}{\geq} B_{1,1}(|A|)
\end{align*}

However, the equality cases in the edge-isoperimetric inequality (see \cite{BIM}) say that 
$$
\inf_{A\subset  \{-1,1\}^{n}, |A|=t}\| |\nabla \mathbbm{1}_{A}|_{1}\|_{1} = B_{1,1}(t), 
$$
therefore, Theorem~\ref{mth02} can be seen as a sharpening of the edge-isoperimetric inequality. 
\item[4.] Stein's theorem (see Theorem 8 in \cite{Stein}) implies for any $p \in (1, \infty)$ existence of a universal constant $C>0$ such that 
\begin{align*}
    \| S_{2}(f)\|_{p} \leq C \max\{p^{1/2}, (p-1)^{-1/2}\} \, \| |\nabla f|_{2} \|_{p}
\end{align*}
holds for all $f : \{-1,1\}^{n} \mapsto \mathbb{R}$ and all $n \geq 1$. 
\end{itemize}

In general, it is an interesting question to find sharp lower bounds on $\| S_{\beta}({\mathbbm{1}}_{A})\|_{p}$ in terms of the size of $|A|$. We have showed in this paper that for certain parameters $\beta$ and $p$ the sharp lower bounds on $\| S_{\beta}(\mathbbm{1}_{A})\|_{p}$ and $\|\, |\nabla \mathbbm{1}_{A}|_{\beta} \,\|_{p}$ sometimes match (up to universal constant factors) and sometimes they are different. 

\subsection*{Acknowledgments}
P.I. was supported in part by NSF CAREER grant DMS-2152401 and a Simons Fellowship. The authors acknowledge the use of AI tools: Brownian-exit-time Bellman function used in the proof of Theorem~\ref{mth01} was suggested by Grok 4.20. An earlier version of the paper instead relied on the Gaussian isoperimetric profile, which is suboptimal and incurs a square-root logarithmic loss relative to the sharp bound.

\section{Proof of Theorem \ref{mth02} and Proposition \ref{col:asympT}} 

In this section we work on {$[0,1)$} equipped with Borel measure.

\begin{lemma}\label{lem:Bellman-lower-11}
Let $Q:[0,1]\to [0,\infty)$ be continuous, with $Q(0)=Q(1)=0$. Assume that
for all $x,y\in[0,1]$,
\begin{equation}\label{eq:Q-twopoint-11}
Q\!\left(\frac{x+y}{2}\right)
\le
\frac{Q(x)+Q(y)}{2}
+
\left|\frac{x-y}{2}\right|.
\end{equation}
Then for every Borel set $A\subset{[0,1)}$,
\[
\|S_1(\mathbbm 1_A)\|_1 \ge Q(|A|).
\]
\end{lemma}

\begin{proof}
Let $f=\mathbbm 1_A\in {L^1([0,1))}$, recall that
\[
f_n=\mathbb E(f\mid \mathcal F_n),\qquad d_n=f_n-f_{n-1},
\]
for $N\ge 0$ set
\[
S_{1,N}(f)=\sum_{n=1}^N |d_n|.
\]
When $N=0$ we set $S_{1, 0}(f)=0$. We will prove by induction on $N$, that for every $N\ge 0$, we have
\begin{equation}\label{eq:finite-Bellman-11}
Q\!\left(\int_0^1 f\right)
\le
\int_0^1 S_{1,N}(f)
+
\int_0^1 Q(f_N).
\end{equation}

For $N=0$, we have $f_0=\int_0^1 f$, so \eqref{eq:finite-Bellman-11} is
immediate.
Assume now that \eqref{eq:finite-Bellman-11} holds for some $N\ge 0$.
Fix a dyadic interval $I\in\mathcal D_N$, and let $I^-,I^+$ be its two dyadic
children. Write
\[
a=\langle f\rangle_{I^-},\qquad b=\langle f\rangle_{I^+}.
\]
Then
\[
f_N|_I=\frac{a+b}{2},\qquad
f_{N+1}|_{I^-}=a,\qquad
f_{N+1}|_{I^+}=b.
\]
Applying \eqref{eq:Q-twopoint-11} to $a$ and $b$, we obtain
\[
Q\!\left(\frac{a+b}{2}\right)
\le
\frac{Q(a)+Q(b)}{2}
+
\left|\frac{a-b}{2}\right|.
\]
Multiplying by $|I|$ gives
\[
\int_I Q(f_N)
\le
\int_I Q(f_{N+1})
+
|I|\left|\frac{a-b}{2}\right|.
\]
On the other hand, since $f_{N+1}-f_N$ is constant on each child of $I$,
\[
\int_I |f_{N+1}-f_N|
=
\frac{|I|}{2}\left|a-\frac{a+b}{2}\right|
+
\frac{|I|}{2}\left|b-\frac{a+b}{2}\right|
=
|I|\left|\frac{a-b}{2}\right|.
\]
Hence
\[
\int_I Q(f_N)
\le
\int_I Q(f_{N+1})
+
\int_I |f_{N+1}-f_N|.
\]
Summing over all $I\in\mathcal D_N$, we get
\[
\int_0^1 Q(f_N)
\le
\int_0^1 Q(f_{N+1})
+
\int_0^1 |f_{N+1}-f_N|.
\]
Combining this with the induction hypothesis yields
\begin{small}
    \begin{align*}
Q\!\left(\int_0^1 f\right)
&\le
\int_0^1 S_{1,N}(f)
+
\int_0^1 Q(f_N) \\
&\le
\int_0^1 S_{1,N}(f)
+
\int_0^1 |f_{N+1}-f_N|
+
\int_0^1 Q(f_{N+1}) \\
&=
\int_0^1 S_{1,N+1}(f)
+
\int_0^1 Q(f_{N+1}),
\end{align*}
\end{small}
which proves \eqref{eq:finite-Bellman-11} for $N+1$.
Now let $N\to\infty$. Since $f=\mathbbm 1_A$, dyadic martingale convergence
gives
\[
f_N\to f
\qquad\text{a.e. on }{[0,1)}.
\]
Because $f\in\{0,1\}$ a.e. and $Q(0)=Q(1)=0$, we have
\[
Q(f_N)\to 0
\qquad\text{a.e. on }{[0,1)}.
\]
Since $Q$ is continuous on $[0,1]$, it is bounded, so dominated convergence
implies
\[
\int_0^1 Q(f_N)\to 0.
\]
Also, $S_{1,N}(f)\uparrow S_1(f)$ pointwise, so by monotone convergence,
\[
\int_0^1 S_{1,N}(f)\to \int_0^1 S_1(f).
\]
Passing to the limit in \eqref{eq:finite-Bellman-11}, we obtain
\[
Q(|A|)
=
Q\!\left(\int_0^1 f\right)
\le
\int_0^1 S_1(f)
=
\|S_1(\mathbbm 1_A)\|_1.
\]
This proves the lemma.
\end{proof}

\begin{proof}[Proof of Theorem \ref{mth02}]
Applying Lemma~\ref{lem:Bellman-lower-11} with $Q=B_{1,1}$, we obtain
\begin{equation}\label{eq:S1-lower-B11-clean}
\|S_1(\mathbbm 1_A)\|_1 \ge B_{1,1}(|A|)
\end{equation}
for every Borel set $A\subset{[0,1)}$.
We next compute the square function explicitly for an interval. Fix
$x\in[0,1]$, and set
\[
A=[0,x),\qquad f=\mathbbm 1_{[0,x)}.
\]
For each $n\ge 1$, let
\[
x\,2^{n-1}=m+r,
\qquad m\in\mathbb Z_{\ge 0},\quad r\in[0,1).
\]
If $r=0$, then $x$ is an endpoint of some dyadic interval of level $n-1$,
so $f$ is constant on each such interval and therefore
\[
d_n=0.
\]
Assume now that $0<r<1$. Then for  dyadic intervals of level $n-1$ on which
$f$ is constant contribute nothing to $d_n$, and the only possible nonzero
contribution comes from the unique dyadic interval
\[
I=\left[\frac{m}{2^{n-1}},\frac{m+1}{2^{n-1}}\right)
\]
whose interior contains $x$.
Let $I^-,I^+$ be the two dyadic children of $I$.

If $0<r\le \frac12$, then
\[
\langle f\rangle_{I^-}=2r,\qquad \langle f\rangle_{I^+}=0,
\qquad \langle f\rangle_I=r.
\]
Hence
\[
\int_I |d_n|
=
\frac{|I|}{2}\,|2r-r|
+
\frac{|I|}{2}\,|0-r|
=
|I|\,r
=
\frac{r}{2^{n-1}}.
\]

If $\frac12<r<1$, then
\[
\langle f\rangle_{I^-}=1,\qquad \langle f\rangle_{I^+}=2r-1,
\qquad \langle f\rangle_I=r.
\]
Therefore
\[
\int_I |d_n|
=
\frac{|I|}{2}\,|1-r|
+
\frac{|I|}{2}\,|2r-1-r|
=
|I|\,(1-r)
=
\frac{1-r}{2^{n-1}}.
\]

Combining the two cases, we obtain
\[
\|d_n\|_1
=
\frac{\min\{r,1-r\}}{2^{n-1}}
=
\frac{\operatorname{dist}(2^{n-1}x,\mathbb Z)}{2^{n-1}},
\qquad n\ge 1.
\]
Thus, by monotone convergence,
\begin{align*}
\|S_1(\mathbbm 1_{[0,x)})\|_1
&=
\int_0^1 \sum_{n=1}^\infty |d_n|
=
\sum_{n=1}^\infty \|d_n\|_1 \\
&=
\sum_{n=1}^\infty
\frac{\operatorname{dist}(2^{n-1}x,\mathbb Z)}{2^{n-1}}
=
\sum_{m=0}^\infty \frac{\operatorname{dist}(2^m x,\mathbb Z)}{2^m}
=
T(x).
\end{align*}
Hence
\[
\inf_{\substack{A\subset{[0,1)}\ \mathrm{Borel}\\ |A|=x}}
\|S_1(\mathbbm 1_A)\|_1
\le
\|S_1(\mathbbm 1_{[0,x)})\|_1
=
T(x).
\]

On the other hand, taking the infimum in \eqref{eq:S1-lower-B11-clean}, we get
\[
\inf_{\substack{A\subset{[0,1)}\ \mathrm{Borel}\\ |A|=x}}
\|S_1(\mathbbm 1_A)\|_1
\ge
B_{1,1}(x).
\]
Since $B_{1,1}(x)=T(x)$ on $[0,1]$ by \cite[Theorem 7]{Lev}, we conclude that
\[
\inf_{\substack{A\subset{[0,1)}\ \mathrm{Borel}\\ |A|=x}}
\|S_1(\mathbbm 1_A)\|_1
=
B_{1,1}(x)
=
T(x),
\]
which proves the theorem.
\end{proof}


\begin{proof}[Proof of Proposition \ref{col:asympT}]
Using the Takagi representation \eqref{eq:Tagaki}
and the symmetry $T(x)=T(1-x)$, it is enough to consider
$
u:=x^*\in(0,1/2].
$
Set
$$
N:=\left\lfloor \log_2\frac{1}{u}\right\rfloor\implies 2^{-N-1}<u\le 2^{-N}.
$$
For $0\le n\le N-1$ we have $2^n u\le 1/2$, and so
$
\operatorname{dist}(2^n u,\mathbb Z)=2^n u.
$
Thus
$$
\sum_{n=0}^{N-1}\frac{\operatorname{dist}(2^n u,\mathbb Z)}{2^n}
=
\sum_{n=0}^{N-1}u
=
Nu.
$$
For the tail, since $\operatorname{dist}(t,\mathbb Z)\le 1/2$ for all $t\in\mathbb R$,
$$
\sum_{n=N}^\infty \frac{\operatorname{dist}(2^n u,\mathbb Z)}{2^n}
\le
\sum_{n=N}^\infty \frac{1/2}{2^n}
=
2^{-N}
\le 2u.
$$
Therefore
$
Nu\le T(u)\le Nu+2u.
$
Since $u=x^*$, and
$
N\le \log_2\frac{1}{u}<N+1.
$
We get
$$
T(x^*)= x^*\log_2\frac{1}{x^*}+O(x^*) .
$$
This finishes the proof of Proposition~\ref{col:asympT}.
\end{proof}



\section{Proof of Theorem \ref{mthgenab} and Remark \ref{supersol}}
\label{sec:proof-mthgenab}

In this section we prove Theorem~\ref{mthgenab}. Thus, for every pair $(\alpha,\beta)$ with $\beta\ge 1\ge \alpha>0$, we will show that $\|S_\beta(\mathbbm{1}_A)\|_\alpha \ge B_{\alpha,\beta}(|A|)$ for every Borel set $A\subset{[0,1)}$. We will also prove
Remark~\ref{supersol} under the standing continuity assumption on the
supersolution.

Recall that for $f\in {L^1([0,1))}$ we define $f_n=\mathbb E(f\mid\mathcal F_n)$ and $d_n=f_n-f_{n-1}$, and for $N\ge1$,
\[
S_\beta(f)=\lim_{N\to\infty} S_{\beta,N}(f)
\quad 
\text{where}
\quad
S_{\beta,N}(f)=\left(\sum_{n=1}^N |d_n|^\beta\right)^{1/\beta},
.
\]

We begin with the finite-level Bellman iteration.

\begin{lemma}\label{lem:51}
Let $Q,M:[0,1]\times[0,\infty)\to\mathbb R$ be Borel measurable
functions such that
\begin{equation}\label{eq:3pi}
Q\!\left(p+a,\,(|a|^\beta+q^\beta)^{1/\beta}\right)
+
Q\!\left(p-a,\,(|a|^\beta+q^\beta)^{1/\beta}\right)
\ge 2Q(p,q)
\end{equation}
for all $p\in[0,1]$, all $q\ge0$, and all $a\in\mathbb R$ with
$p\pm a\in[0,1]$. Assume also that
\[
Q(p,q)\le M(p,q)
\qquad \text{for all } p\in[0,1],\ q\ge0.
\]
Then for every $N\ge1$ and every $\mathcal F_N$-measurable function
$f:[0,1]\to[0,1]$,
\begin{equation}\label{eq:finite-level-bellman}
Q\!\left(\int_0^1 f,\ q\right)
\le
\int_0^1
M\!\left(f,\,(S_{\beta,N}(f)^\beta+q^\beta)^{1/\beta}\right).
\end{equation}
\end{lemma}

\begin{proof}
For $0\le n\le N$, define
\[
\Phi_n
=
\sum_{I\in\mathcal D_n}
|I|\,Q\!\left(\langle f\rangle_I,\ q_I^{(n)}\right),
\]
where
\[
q_I^{(n)}
=
\left(
q^\beta+\sum_{k=1}^n |d_k|_I^\beta
\right)^{1/\beta},
\]
and $|d_k|_I$ denotes the constant value of $|d_k|$ on $I$. This is
well-defined because $d_k$ is $\mathcal F_k$-measurable and hence is
constant on every $I\in\mathcal D_n$ whenever $k\le n$.

For $n=0$ we have
\[
\Phi_0=Q\!\left(\int_0^1 f,\ q\right).
\]

We claim that
\[
\Phi_{n-1}\le \Phi_n
\qquad \text{for every } 1\le n\le N.
\]
Fix $I\in\mathcal D_{n-1}$, and write $I^-,I^+$ for its two dyadic
children. Set
\[
p_I=\langle f\rangle_I,
\qquad
a_I=\frac{\langle f\rangle_{I^+}-\langle f\rangle_{I^-}}{2}.
\]
Then
\[
\langle f\rangle_{I^-}=p_I-a_I,
\qquad
\langle f\rangle_{I^+}=p_I+a_I.
\]
Moreover, on $I^-$ and $I^+$ we have
\[
d_n=-a_I \quad \text{and} \quad d_n=a_I,
\]
respectively. Hence
\[
q_{I^-}^{(n)}=q_{I^+}^{(n)}
=
\left((q_I^{(n-1)})^\beta+|a_I|^\beta\right)^{1/\beta}.
\]
Applying \eqref{eq:3pi} with $p=p_I$, $a=a_I$, and $q=q_I^{(n-1)}$, we obtain
\begin{align*}
|I|\,Q(p_I,q_I^{(n-1)})
&\le
\frac{|I|}{2}\,
Q\!\left(p_I-a_I,\ q_{I^-}^{(n)}\right)
+
\frac{|I|}{2}\,
Q\!\left(p_I+a_I,\ q_{I^+}^{(n)}\right) \\
&=
|I^-|\,Q\!\left(\langle f\rangle_{I^-}, q_{I^-}^{(n)}\right)
+
|I^+|\,Q\!\left(\langle f\rangle_{I^+}, q_{I^+}^{(n)}\right).
\end{align*}
Summing over all $I\in\mathcal D_{n-1}$ gives $\Phi_{n-1}\le \Phi_n$.
Therefore
\[
Q\!\left(\int_0^1 f,\ q\right)=\Phi_0\le \Phi_N.
\]
Since $f$ is $\mathcal F_N$-measurable, it is constant on each
$J\in\mathcal D_N$, so $\langle f\rangle_J=f|_J$. Also,
\[
q_J^{(N)}
=
\left(S_{\beta,N}(f)^\beta+q^\beta\right)^{1/\beta}
\quad \text{on } J.
\]
Thus, using $Q\le M$,
\begin{small}
    \begin{align*}
\Phi_N
&=
\sum_{J\in\mathcal D_N}
|J|\,
Q\!\left(f|_J,\ (S_{\beta,N}(f)^\beta+q^\beta)^{1/\beta}\right) \\
&\le
\sum_{J\in\mathcal D_N}
|J|\,
M\!\left(f|_J,\ (S_{\beta,N}(f)^\beta+q^\beta)^{1/\beta}\right) =
\int_0^1
M\!\left(f,\ (S_{\beta,N}(f)^\beta+q^\beta)^{1/\beta}\right).
\end{align*}
\end{small}
This proves \eqref{eq:finite-level-bellman}.
\end{proof}

We now verify that the Bellman candidate \eqref{eq:QB} satisfies the required three-point inequality.

\begin{lemma}\label{lem:u3pt}
Let $\beta\ge1\ge\alpha>0$, and let $B:[0,1]\to[0,\infty)$ satisfy the
two-point inequality \eqref{twopp} and the boundary condition
\[
B(0)=B(1)=0.
\]
Define
\begin{equation}\label{eq:QB}
Q_B(p,q)=\left(B(p)^\beta+q^\beta\right)^{\alpha/\beta}.
\end{equation}
Then $Q_B$ satisfies the three-point inequality \eqref{eq:3pi}. Moreover,
\[
Q_B(0,q)=Q_B(1,q)=q^\alpha
\qquad \text{for all } q\ge0.
\]
\end{lemma}

\begin{proof}
The endpoint identities are immediate from $B(0)=B(1)=0$.
Set
\[
r=\frac{\alpha}{\beta}\in(0,1].
\]
Fix $p\in[0,1]$, $q\ge0$, and $a\in\mathbb R$ such that $p\pm a\in[0,1]$.
Let
\[
u=B(p+a)^\beta+|a|^\beta,
\qquad
v=B(p-a)^\beta+|a|^\beta.
\]
By the two-point inequality \eqref{twopp}, written with $x=p+a$ and
$y=p-a$, we have
\[
\frac{u^r+v^r}{2}\ge B(p)^\alpha.
\]

If $r=1$, then $\alpha=\beta$, and therefore
\[
\frac{u+v}{2}\ge B(p)^\beta.
\]
Adding $q^\beta$ to both sides yields
\[
\frac{(u+q^\beta)+(v+q^\beta)}{2}
\ge
B(p)^\beta+q^\beta,
\]
which is exactly \eqref{eq:3pi} for $Q_B$.

Assume now that $0<r<1$. Define
\[
\psi(t)=\left(\frac{(u+t)^r+(v+t)^r}{2}\right)^{1/r}-t,
\qquad t\ge0.
\]
We claim that $\psi$ is nondecreasing on $[0,\infty)$.
For $t>0$, let $Y$ be the random variable taking the values $u+t$ and
$v+t$ with probability $1/2$. Then
\[
\psi(t)=\left(\mathbb E Y^r\right)^{1/r}-t,
\]
and differentiation gives
\[
\psi'(t)
=
\left(\mathbb E Y^r\right)^{1/r-1}\mathbb E Y^{r-1}-1.
\]
Since $Y>0$ almost surely and the power means are monotone in the exponent,
\[
\left(\mathbb E Y^r\right)^{1/r}
\ge
\left(\mathbb E Y^{r-1}\right)^{1/(r-1)}.
\]
Because $r-1<0$, raising both sides to the power $r-1$ reverses the
inequality, and we obtain
\[
\left(\mathbb E Y^r\right)^{(r-1)/r}
\le
\mathbb E Y^{r-1}.
\]
Therefore
\[
\left(\mathbb E Y^r\right)^{1/r-1}\mathbb E Y^{r-1}\ge1,
\]
so $\psi'(t)\ge0$ for all $t>0$. Hence $\psi$ is nondecreasing on
$[0,\infty)$.
Taking $t=q^\beta$ and using the inequality at $t=0$, we get
\[
\psi(q^\beta)\ge \psi(0)\ge B(p)^\beta.
\]
That is,
\[
\left(
\frac{(u+q^\beta)^r+(v+q^\beta)^r}{2}
\right)^{1/r}
\ge
B(p)^\beta+q^\beta.
\]
Raising both sides to the power $r$ yields
\[
\frac{(u+q^\beta)^r+(v+q^\beta)^r}{2}
\ge
\left(B(p)^\beta+q^\beta\right)^r,
\]
which is precisely \eqref{eq:3pi} for $Q_B$.
\end{proof}

\begin{proof}[Proof of Theorem~\ref{mthgenab} and Remark~\ref{supersol} in the continuous case]
Let $B:[0,1]\to[0,\infty)$ be continuous, satisfy the two-point
inequality \eqref{twopp}, and satisfy
\[
B(0)=B(1)=0.
\]
Define
\[
Q_B(p,q)=\left(B(p)^\beta+q^\beta\right)^{\alpha/\beta}.
\]
By Lemma~\ref{lem:u3pt}, the function $Q_B$ satisfies the three-point
inequality \eqref{eq:3pi}.

Let $A\subset{[0,1)}$ be Borel, and set
\[
f=\mathbbm 1_A,
\qquad
f_N=\mathbb E(f\mid\mathcal F_N).
\]
Then $f_N$ is $\mathcal F_N$-measurable, $0\le f_N\le1$, and
\[
\int_0^1 f_N=\int_0^1 f=|A|.
\]
Also, the first $N$ martingale differences of $f_N$ agree with those of
$f$, and therefore
\[
S_{\beta,N}(f_N)=S_{\beta,N}(f).
\]

Applying Lemma~\ref{lem:51} with $Q=M=Q_B$ and $q=0$, we obtain
\begin{small}
    \begin{align*}
B(|A|)^\alpha
=
Q_B(|A|,0)
&\le
\int_0^1
Q_B\!\left(f_N,S_{\beta,N}(f)\right) =
\int_0^1
\left(B(f_N)^\beta + S_{\beta,N}(f)^\beta\right)^{\alpha/\beta}.
\end{align*}
\end{small}
Since $0<\alpha/\beta\le1$, we have
\[
(u+v)^{\alpha/\beta}\le u^{\alpha/\beta}+v^{\alpha/\beta}
\qquad \text{for all } u,v\ge0.
\]
Hence
\[
B(|A|)^\alpha
\le
\int_0^1 B(f_N)^\alpha
+
\int_0^1 S_{\beta,N}(f)^\alpha.
\]

Now $f_N\to f=\mathbbm 1_A$ almost everywhere, and since $0\le f_N\le1$,
$B$ is continuous on $[0,1]$ and vanishes at $0$ and $1$, dominated
convergence gives
\[
\lim_{N\to\infty}\int_0^1 B(f_N)^\alpha = 0.
\]
Also, $S_{\beta,N}(f)\uparrow S_\beta(f)$ pointwise, so by monotone
convergence,
\[
\lim_{N\to\infty}\int_0^1 S_{\beta,N}(f)^\alpha
=
\int_0^1 S_\beta(f)^\alpha
=
\|S_\beta(\mathbbm 1_A)\|_\alpha^\alpha.
\]
Therefore
$
B(|A|)^\alpha \le \|S_\beta(\mathbbm 1_A)\|_\alpha^\alpha,
$
and hence
\[
\|S_\beta(\mathbbm 1_A)\|_\alpha \ge B(|A|).
\]

Taking $B=B_{\alpha,\beta}$ proves Theorem~\ref{mthgenab}. The same
argument applies to every continuous supersolution
$\widetilde B_{\alpha,\beta}$ satisfying the two-point inequality
\eqref{twopp} and the boundary conditions
\[
\widetilde B_{\alpha,\beta}(0)=\widetilde B_{\alpha,\beta}(1)=0.
\]
This proves Remark~\ref{supersol} in the continuous case.
\end{proof}

This finishes the proof of Theorem~\ref{mthgenab}. In particular, taking
$B=B_{1,1}$ and using the identification $B_{1,1}=T$, we obtain
\[
\|S_1(\mathbbm 1_A)\|_1 \ge T(|A|)\ge |A|^*\log_2\!\left(\frac{1}{|A|^*}\right),
\]
where the last inequality follows from Theorem~\ref{thm-bn} part $\textup{(v)}$ 
(this can also be found in \cite{Lev}).
Likewise, taking $B=B_{1,2}=I$, we obtain
\[
\|S_2(\mathbbm 1_A)\|_1 \ge I(|A|)\asymp |A|^*\sqrt{\log(1/|A|^*)}.
\]
The sharper estimate of Theorem~\ref{mth01} is proved in
Section~\ref{sec:proof-mth01}.

\section{Proof of Theorem \ref{mth01} and Proposition \ref{prop:const}}\label{sec:proof-mth01} 

In this section we complete the proof of Theorem~\ref{mth01} in the Borel-set setting. We continue to work with the dyadic filtration $(\mathcal F_n)$ on {$[0,1)$}, and for a Borel set $A\subset{[0,1)}$ we define $S_2(\mathbbm 1_A)$ through the dyadic martingale generated by $\mathbbm 1_A$.

We will reuse the tree-iteration framework from Section~\ref{sec:proof-mthgenab}. 
The difference in the present case is that our Bellman function only satisfies an endpoint obstacle condition, rather than a global obstacle condition (due to the fact that $f$ is boolean). 
Accordingly, we first show the following dyadic endpoint-obstacle variant of Lemma~\ref{lem:51}.

\begin{corollary}\label{cor:endpoint-lemma}
Let $Q:[0,1]\times[0,\infty)\to\mathbb R$ be a Borel measurable function satisfying 
\begin{equation}\label{eq:3pi-S2}
Q\!\left(p+a,\sqrt{q^{2}+a^{2}}\right)
+
Q\!\left(p-a,\sqrt{q^{2}+a^{2}}\right)
\ge 2Q(p,q)
\end{equation}
for all $p\in[0,1]$, all $q\ge0$, and all $a\in\mathbb R$ such that $p\pm a\in[0,1]$. Let $\phi:[0,\infty)\to\mathbb R$ be a Borel measurable function satisfying 
\begin{equation}\label{eq:endpoint-obstacle}
Q(0,q)\le \phi(q^2),\qquad Q(1,q)\le \phi(q^2)
\qquad\text{for all }q\ge0.
\end{equation}
Then for every $A\in\mathcal D$,
\[
Q(|A|,0)\le \int_0^1 \phi\!\big(S_2(\mathbbm 1_A)^2\big).
\]
\end{corollary}

\begin{proof}
The proof is the same tree iteration as in Lemma~\ref{lem:51}, with $f=\mathbbm 1_A$. 
Since $A\in\mathcal D$, there exists $N$ such that $f$ is $\mathcal F_N$-measurable. 
Applying the proof of Lemma~\ref{lem:51} with $\beta=2$ and $q=0$, we obtain
\[
Q(|A|,0)\le \Phi_N,
\]
where $\Phi_N$ is the terminal Bellman sum over dyadic atoms of level $N$.

At the terminal level $N$, the function $f$ takes only the values $0$ and $1$ on dyadic atoms, and moreover $d_n=0$ for all $n>N$, so
\[
S_2(f)=S_{2,N}(f).
\]
Thus, on each dyadic atom $J\in\mathcal D_N$, the endpoint obstacle condition \eqref{eq:endpoint-obstacle} gives
\[
Q\!\left(f|_J,\ S_2(f)|_J\right)\le \phi\!\left(S_2(f)|_J^2\right).
\]
Summing over all such $J$, we obtain
\[
\Phi_N\le \int_0^1 \phi\!\big(S_2(\mathbbm 1_A)^2\big),
\]
and therefore
\[
Q(|A|,0)\le \int_0^1 \phi\!\big(S_2(\mathbbm 1_A)^2\big),
\]
as claimed.
\end{proof}





We now construct the Bellman function needed for the sharp lower bound.
Let $(B_t)_{t\ge0}$ be a standard Brownian motion started at $B_0=p\in[0,1]$, and
\[
\tau:=\inf\{t\ge0:\, B_t\notin(0,1)\}
\]
be its first exit time from $(0,1)$. \\

Let $\phi:[0,\infty)\to\mathbb R$ be concave, such that for every $r\in[0,1]$ and $s\ge 0$
\begin{equation}
\label{eq:abs-int}
    \mathbb E_r\!\bigl[\,|\phi(s+\tau)|\,\bigr]<\infty.
\end{equation}
Set
\begin{equation}\label{eq:Utilde-def-alpha}
U_\phi(p,q):=\mathbb E_p\!\big[\phi(q^2+\tau)\big], \qquad p\in[0,1],\ q\ge 0.
\end{equation}

\begin{lemma}\label{lem:one-step-concavity}
Let $\phi:[0,\infty)\to\mathbb R$ be concave satisfing \eqref{eq:abs-int}.
Let $U_\phi$ be defined by \eqref{eq:Utilde-def-alpha}. Then
\[
U_\phi(0,q)=U_\phi(1,q)=\phi(q^2), \qquad q\ge 0,
\]
and for every $p\in[0,1]$, every $q\ge 0$, and every $a\in\mathbb R$ such that
$p\pm a\in[0,1]$,
\[
U_\phi\!\left(p+a,\sqrt{q^2+a^2}\right)
+
U_\phi\!\left(p-a,\sqrt{q^2+a^2}\right)
\ge 2\,U_\phi(p,q).
\]
\end{lemma}

\begin{proof}
The boundary condition is immediate. Indeed for $p\in\{0,1\}$, then $\tau=0$ almost
surely, hence
\[
U_\phi(p,q)=\mathbb E_p[\phi(q^2+\tau)]=\phi(q^2).
\]

To prove the inequality, note that it is invariant under $a\mapsto -a$, so it
is enough to consider $a\ge 0$. The case $a=0$ is trivial, so assume $a>0$.

Define
\[
\sigma:=\inf\{t\ge 0:\,|B_t-p|=a\},
\quad
A_+:=\{B_\sigma=p+a\},
\quad
A_-:=\{B_\sigma=p-a\}.
\]
Because $p\pm a\in[0,1]$ and Brownian paths are continuous, the process must hit
one of the points $p-a$ or $p+a$ before it can exit $(0,1)$. So $\sigma\le \tau$ a.s.

Define the martingale $M_t:=(B_t-p)^2-t$.
By optional stopping at $t\wedge \sigma$,
\[
\mathbb E_p[M_{t\wedge \sigma}]=\mathbb E_p[M_0]=0,
\]
so
\[
\mathbb E_p[t\wedge \sigma]
=
\mathbb E_p[(B_{t\wedge \sigma}-p)^2].
\]
Since $|B_{t\wedge \sigma}-p|\le a$, the right-hand side is bounded by $a^2$.
Letting $t\to\infty$ and using monotone convergence on the left and dominated
convergence on the right, we obtain
\[
\mathbb E_p[\sigma]
=
\mathbb E_p[(B_\sigma-p)^2]
=
a^2.
\]

Now we use the symmetry of Brownian motion when it first hits one of the two
points $p-a$ or $p+a$. Under $\mathbb P_p$, the shifted process
\[
X_t:=B_t-p
\]
is standard Brownian motion starting at $0$. Brownian motion is symmetric, so
$X_t$ and $-X_t$ have the same distribution. Also, replacing $X$ by $-X$ does
not change the stopping time
\[
\sigma=\inf\{t\ge 0:\,|X_t|=a\},
\]
because $\sigma$ only depends on $|X_t|$. On the other hand, this reflection
swaps the two events of first hitting $p\pm a$. In particular, for every nonnegative measurable function
$h$,
\[
\mathbb E_p[h(\sigma)\mathbbm 1_{A_+}]
=
\mathbb E_p[h(\sigma)\mathbbm 1_{A_-}]
=
\frac12\,\mathbb E_p[h(\sigma)].
\]

By applying this identity to the positive and negative parts, the same formula
holds for every integrable function $h$.

For each $r\in[0,1]$, define
\[
G_r(s):=\mathbb E_r[\phi(s+\tau)],
\qquad
H_r(s):=\mathbb E_r[|\phi(s+\tau)|],
\qquad s\ge 0.
\]
These are finite by assumption. Moreover, $G_r$ is concave on $[0,\infty)$.
Indeed, if $s_1,s_2\ge 0$ and $\theta\in[0,1]$, then concavity of $\phi$ gives
\[
\phi\!\left(\theta(s_1+\tau)+(1-\theta)(s_2+\tau)\right)
\ge
\theta\,\phi(s_1+\tau)+(1-\theta)\,\phi(s_2+\tau)
\]
almost surely. Taking $\mathbb E_r$ yields
\[
G_r(\theta s_1+(1-\theta)s_2)
\ge
\theta G_r(s_1)+(1-\theta)G_r(s_2).
\]

By the strong Markov property at time $\sigma$,
\[
\mathbb E_p\!\left[\,|\phi(q^2+\tau)|\,\middle|\,\mathcal F_\sigma\right]
=
H_{B_\sigma}(q^2+\sigma)
=
\mathbbm 1_{A_+}H_{p+a}(q^2+\sigma)
+
\mathbbm 1_{A_-}H_{p-a}(q^2+\sigma).
\]
Taking expectations yields
\[
\mathbb E_p\!\bigl[\mathbbm 1_{A_+}H_{p+a}(q^2+\sigma)\bigr]
+
\mathbb E_p\!\bigl[\mathbbm 1_{A_-}H_{p-a}(q^2+\sigma)\bigr]
=
\mathbb E_p\!\bigl[|\phi(q^2+\tau)|\bigr]
<\infty.
\]
Applying the reflection identity twice with
$h(\sigma)=H_{p\pm a}(q^2+\sigma)$, we obtain
\[
\mathbb E_p[H_{p\pm a}(q^2+\sigma)]<\infty.
\]
Since $|G_r|\le H_r$, it follows that
\[
G_{p+a}(q^2+\sigma),\qquad G_{p-a}(q^2+\sigma)
\]
are integrable.

Now apply the strong Markov property at time $\sigma$ to $\phi(q^2+\tau)$,
\[
\mathbb E_p\!\left[\,\phi(q^2+\tau)\,\middle|\,\mathcal F_\sigma\right]
=
G_{B_\sigma}(q^2+\sigma)
=
\mathbbm 1_{A_+}G_{p+a}(q^2+\sigma)
+
\mathbbm 1_{A_-}G_{p-a}(q^2+\sigma).
\]
Taking expectations and using the reflection identity for the integrable
functions $h(\sigma)=G_{p+a}(q^2+\sigma)$ and
$h(\sigma)=G_{p-a}(q^2+\sigma)$, we get
\[
U_\phi(p,q)
=
\frac12\,\mathbb E_p\!\bigl[G_{p+a}(q^2+\sigma)\bigr]
+
\frac12\,\mathbb E_p\!\bigl[G_{p-a}(q^2+\sigma)\bigr].
\]

Since each $G_r$ is concave and $\mathbb E_p[\sigma]=a^2$, Jensen's inequality
gives
\[
\mathbb E_p\!\bigl[G_r(q^2+\sigma)\bigr]
\le
G_r(q^2+\mathbb E_p[\sigma])
=
G_r(q^2+a^2).
\]
By definition of $G_r$,
\[
G_r(q^2+a^2)
=
\mathbb E_r[\phi(q^2+a^2+\tau)]
=
U_\phi\!\left(r,\sqrt{q^2+a^2}\right).
\]
Applying this with $r=p+a$ and $r=p-a$, we conclude that
\[
U_\phi(p,q)
\le
\frac12\,U_\phi\!\left(p+a,\sqrt{q^2+a^2}\right)
+
\frac12\,U_\phi\!\left(p-a,\sqrt{q^2+a^2}\right).
\]
Multiplying by $2$ proves the claim.
\end{proof}

As an immediate consequence of Corollary~\ref{cor:endpoint-lemma}, we obtain the lower bound for dyadic sets.

\begin{corollary}\label{cor:dyadic-lowerbound-S2}
For $\phi(t) = \sqrt{t}$, define
\begin{equation}\label{eq:Utilde-def}
\tilde U(p,q):=\mathbb E_p\sqrt{q^2+\tau},
\qquad p\in[0,1],\ q\ge0.
\end{equation}
For every $A\in\mathcal D$,
\[
\|S_2(\mathbbm 1_A)\|_1\ge \tilde U(|A|,0)=\mathbb E_{|A|}\sqrt{\tau}.
\]
\end{corollary}

\begin{proof}
Apply Corollary~\ref{cor:endpoint-lemma} with $Q=\tilde U$, and use Lemma~\ref{lem:one-step-concavity}.
\end{proof}

We now pass from dyadic sets to arbitrary Borel sets.

\begin{proposition}\label{prop:borel-extension-S2}
For every Borel set $A\subset{[0,1)}$,
\[
\|S_2(\mathbbm 1_A)\|_1\ge \tilde U(|A|,0).
\]
\end{proposition}

\begin{proof}
Let $f=\mathbbm 1_A$, and define
\[
f_n:=\mathbb E(f\mid\mathcal F_n),
\qquad
A_n:=\{f_n>1/2\}.
\]
Then $A_n\in\mathcal F_n\subset\mathcal D$. Since $f_n\to f$ almost everywhere and $f\in\{0,1\}$ almost everywhere, we have
\[
\mathbbm 1_{A_n}\to \mathbbm 1_A
\qquad\text{almost everywhere.}
\]
Hence $
|A_n\triangle A|
=
\|\mathbbm 1_{A_n}-\mathbbm 1_A\|_1
\to0
$
by dominated convergence.
By the pointwise reverse triangle inequality in $\ell^2$,
\[
|S_2(u)-S_2(v)|\le S_2(u-v).
\]
Therefore
\[
\big|\|S_2(\mathbbm 1_{A_n})\|_1-\|S_2(\mathbbm 1_A)\|_1\big|
\le
\|S_2(\mathbbm 1_{A_n}-\mathbbm 1_A)\|_1.
\]
By Cauchy--Schwarz and orthogonality of martingale differences,
\[
\|S_2(h)\|_1\le \|S_2(h)\|_2=\|h-\mathbb Eh\|_2\le \|h\|_2.
\]
Applying this to $h=\mathbbm 1_{A_n}-\mathbbm 1_A$, we get
\[
\big|\|S_2(\mathbbm 1_{A_n})\|_1-\|S_2(\mathbbm 1_A)\|_1\big|
\le
|A_n\triangle A|^{1/2}\to0.
\]
Since $A_n\in\mathcal D$, Corollary~\ref{cor:dyadic-lowerbound-S2} gives
\[
\|S_2(\mathbbm 1_{A_n})\|_1\ge \tilde U(|A_n|,0).
\]
Thus it remains to note that $|A_n|\to |A|$ and that $p\mapsto \tilde U(p,0)$ is continuous, so passing to the limit yields
\[
\|S_2(\mathbbm 1_A)\|_1\ge \tilde U(|A|,0).
\]
\end{proof}

\begin{remark}\label{rem:possible-alpha-generalization-S2}
For every $0<\alpha\le 2$, the function
$
\phi_\alpha(t)=t^{\alpha/2}
$
is increasing and concave on $[0,\infty)$.
Hence the candidate function
$$
U_\alpha(p,q):=\mathbb E_p\big(q^2+\tau\big)^{\alpha/2}
$$
satisfies the same one-step inequality
$$
U_\alpha\!\left(p+a,\sqrt{q^2+a^2}\right)
+
U_\alpha\!\left(p-a,\sqrt{q^2+a^2}\right)
\ge 2U_\alpha(p,q),
$$
together with the boundary condition
$$
U_\alpha(0,q)=U_\alpha(1,q)=q^\alpha.
$$
Therefore Corollary~\ref{cor:endpoint-lemma} gives the Brownian lower bound in
$L^\alpha$ for the full range $0<\alpha\le 2$:
$$
\|S_2(\mathbbm 1_A)\|_\alpha^\alpha \ge \mathbb E_{|A|}\tau^{\alpha/2}
\qquad \text{for every } A\in\mathcal D.
$$
The lower bound extends from dyadic sets to arbitrary Borel sets
$A\subset[0,1)$ for every $0<\alpha\le2$. Indeed, if $A_n\in\mathcal D$
approximates $A$ as in Proposition~\ref{prop:borel-extension-S2}, then
$$
\|S_2(\mathbbm 1_{A_n})-S_2(\mathbbm 1_A)\|_\alpha
\le
\|S_2(\mathbbm 1_{A_n}-\mathbbm 1_A)\|_2 \to 0,
$$
since $\alpha\le2$, and $p\mapsto \mathbb E_p\tau^{\alpha/2}$ is continuous
on $[0,1]$. Therefore
$$
\|S_2(\mathbbm 1_A)\|_\alpha^\alpha \ge \mathbb E_{|A|}\tau^{\alpha/2}
\qquad \text{for every Borel set } A\subset[0,1).
$$
Moreover one could show that this bound is sharp by similar steps as in Lemma~\ref{lem:dyadic-mass-upper-S2} below.
\end{remark}

We next show the explicit formula and asymptotic behavior of $\tilde U(p,q)$.
\begin{lemma}\label{lem:Utilde-asymptotic-S2}
We have
\begin{equation}
\label{eq:U-exten-sin}
\tilde U(p,0)=\mathbb E_p\sqrt{\tau}
=
\frac{2\sqrt{2\pi}}{\pi^2}\sum_{k=0}^\infty
\frac{\sin((2k+1)\pi p)}{(2k+1)^2},
\end{equation}
and
\[
\mathbb E_p\sqrt{\tau}\asymp p^*\log_2\frac1{p^*},
\qquad p^*:=\min\{p,1-p\}.
\]
In particular, the map $p\mapsto \tilde U(p,0)$ is continuous on $[0,1]$.
\end{lemma}

\begin{proof}
Let
$
u(t,p):=\mathbb P_p(\tau>t).
$
Then
\[
\mathbb E_p\sqrt{\tau}
=
\frac12\int_0^\infty t^{-1/2}u(t,p)\,dt.
\]
The function $u$ solves the Dirichlet heat problem on $(0,1)$ with initial data $u(0,p)=1$.
(This follows from the correspondence between killed Brownian motion and the Dirichlet heat equation; see \cite[Theorem 7.44 and Theorem 7.45,]{MortersPeresBM}.)
Hence
\[
u(t,p)=\frac4\pi\sum_{k=0}^\infty \frac{\sin((2k+1)\pi p)}{2k+1}
\exp\!\left(-\frac{(2k+1)^2\pi^2}{2}t\right).
\]
Moreover,
\begin{small}
    \begin{align*}
\sum_{k=0}^\infty \frac4\pi \frac{1}{2k+1}
\int_0^\infty t^{-1/2}
\exp\!\left(-\frac{(2k+1)^2\pi^2}{2}t\right)\,dt
&=
\frac{2\sqrt{2\pi}}{\pi^2}
\sum_{k=0}^\infty \frac{1}{(2k+1)^2}
<\infty.
\end{align*}
\end{small}
Thus the series may be integrated termwise, which gives
\[
\mathbb E_p\sqrt{\tau}
=
\frac{2\sqrt{2\pi}}{\pi^2}\sum_{k=0}^\infty
\frac{\sin((2k+1)\pi p)}{(2k+1)^2}.
\]
This series is absolutely and uniformly convergent in $p$, and therefore $p\mapsto \tilde U(p,0)$ is continuous.
To obtain the asymptotic, set $x=\pi p$. Then
\begin{equation}
\label{eq:S(p)-rep}
\sum_{k=0}^\infty \frac{\sin((2k+1)x)}{(2k+1)^2}
=
\mathrm{Cl}_2(x)-\frac14\mathrm{Cl}_2(2x),
\end{equation}
where
\begin{equation}
\label{eq:cl2}
\mathrm{Cl}_2(x):=\sum_{n=1}^\infty \frac{\sin(nx)}{n^2}
=-\int_0^x \log\!\bigl(2\sin(t/2)\bigr)\,dt
\end{equation}
is the Clausen function of order two; see \cite{Lewin1981}.
Using the expansion
\[
\log\!\bigl(2\sin(t/2)\bigr)=\log t+O(t^2)
\qquad (t\downarrow0),
\]
we obtain
\[
\mathrm{Cl}_2(x)=x\log\frac1x+O(x).
\]
Therefore
\begin{equation}
\label{eq:S(p)-approx}
\sum_{k=0}^\infty \frac{\sin((2k+1)x)}{(2k+1)^2}
=
\frac{x}{2}\log\frac1x+O(x),
\end{equation}
and consequently
\[
\mathbb E_p\sqrt{\tau}
=
\sqrt{\frac2\pi}\,p\log\frac1p+O(p)
\qquad (p\downarrow0).
\]
By symmetry under $p\mapsto1-p$, the same estimate holds near $1$, which proves
\[
\mathbb E_p\sqrt{\tau}\asymp p^*\log_2\frac1{p^*},
\]
because when $p$ is separate from the endpoints of the interval $[0,1]$ then both functions are positive,  separated form zero and infinity. 
\end{proof}

\begin{lemma}\label{lem:Utilde-sharp-constants-S2}
Let
$
p^*:=\min\{p,1-p\}.
$ Then for every $p\in[0,1]$,
\[
\sqrt{\frac2\pi}\,\ln 2 \; p^*\log_2\frac1{p^*}
\;\le\;
\tilde U(p,0)
\;\le\;
\frac{4\sqrt2\,G}{\pi^{3/2}}\,p^*\log_2\frac1{p^*},
\]
where
$
G=\sum_{k=0}^\infty \frac{(-1)^k}{(2k+1)^2}
$
is Catalan's constant.
Both constants are sharp:
\[
\lim_{p^*\downarrow0}
\frac{\tilde U(p,0)}{p^*\log_2\frac1{p^*}}
=
\sqrt{\frac2\pi}\,\ln 2,
\]
and equality in the upper bound holds at $p=\frac12$.
\end{lemma}

\begin{proof}
By Lemma~\ref{lem:Utilde-asymptotic-S2}, we can write
\[
\tilde U(p,0)=\mathbb E_p\sqrt{\tau}
=
\frac{2\sqrt2}{\pi^{3/2}}\,S(p),
\quad
\text{where}
\quad 
S(p):=\sum_{k=0}^\infty \frac{\sin((2k+1)\pi p)}{(2k+1)^2}.
\]
By equation \eqref{eq:S(p)-rep} using $x=\pi p$, we have 
\[
S(p)=\mathrm{Cl}_2(\pi p)-\frac14 \mathrm{Cl}_2(2\pi p),
\]
where $\mathrm{Cl}_2(x)$ is given by \eqref{eq:cl2}.
Moreover, since 
$
\sin((2k+1)\pi(1-p))
=
\sin((2k+1)\pi p)
$
for every $k\ge 0$, we have
\[
S(1-p)=S(p), \qquad p\in[0,1].
\]
Hence it is enough to work on $[0,1/2]$, where $p^*=p$.

Differentiating the Clausen representation on $(0,1)$ gives
\[
S'(p)
=
-\pi\log\!\Bigl(2\sin\frac{\pi p}{2}\Bigr)
+\frac{\pi}{2}\log\!\bigl(2\sin(\pi p)\bigr).
\]
Using
$
\sin(\pi p)=2\sin\frac{\pi p}{2}\cos\frac{\pi p}{2},
$
we obtain
\[
S'(p)
=
\frac{\pi}{2}\log\!\Bigl(\cot\frac{\pi p}{2}\Bigr).
\]
Since $S(0)=0$, integration yields
\[
S(p)=\int_0^{\pi p}\frac12\log\!\Bigl(\cot\frac t2\Bigr)\,dt.
\]
Differentiating once more, and using
\[
\frac{d}{dx}\log\!\Bigl(\cot\frac x2\Bigr)=-\frac1{\sin x},
\]
we get
\[
S''(p)=-\frac{\pi^2}{2\sin(\pi p)},
\qquad 0<p<1.
\]

We split the proof into three steps.\\

\textbf{Step 1: Lower bound.}
For $0<p\le 1/2$,
$
L(p):=S(p)-\frac{\pi}{2}p\log\frac1p.
$
Then
\[
L''(p)
=
-\frac{\pi^2}{2\sin(\pi p)}+\frac{\pi}{2p}.
\]
Since $\sin(\pi p)\le \pi p$ on $[0,1/2]$, it follows that
\[
-\frac{\pi^2}{2\sin(\pi p)}\le -\frac{\pi}{2p}\implies L''(p)\le 0
\]
 and so
$L$ is concave on $(0,1/2]$.

Taking $ p\downarrow0$ and using \eqref{eq:S(p)-approx} we get
$
L(p)\to 0
$ as $ (p\downarrow0).
$
Also,
\[
L\Bigl(\frac12\Bigr)
=
S\Bigl(\frac12\Bigr)-\frac{\pi}{4}\log2
=
G-\frac{\pi}{4}\log2
>0.
\]
Therefore the continuous extension of $L$ to $[0,1/2]$ is concave, equals $0$ at $0$, and is positive at $1/2$.
Therefore
\[
L(p)\ge 0
\qquad \text{for all } p\in\Bigl[0,\frac12\Bigr].
\]
Thus 
by symmetry,
\[
S(p)\ge \frac{\pi}{2}p^*\log\frac1{p^*},
\qquad 0\le p\le 1.
\]
Multiplying by $\frac{2\sqrt2}{\pi^{3/2}}$ and using
$
\log\frac1{p^*}=(\ln2)\log_2\frac1{p^*},
$
we obtain
\[
\tilde U(p,0)\ge \sqrt{\frac2\pi}\,\ln2 \; p^*\log_2\frac1{p^*}.
\]

\textbf{Step 2: Upper bound.}
Set
$
A:=\frac{2G}{\ln2},
$
and for $0<p\le 1/2$, define
\[
V(p):=S(p)-A\,p\log\frac1p,
\quad 
\text{then}
\quad 
V''(p)
=
-\frac{\pi^2}{2\sin(\pi p)}+\frac{A}{p}.
\]
Now $\sin(\pi p)\ge 2p$ on $[0,1/2]$, so
\[
-\frac{\pi^2}{2\sin(\pi p)}
\ge
-\frac{\pi^2}{4p}
\quad 
\implies
\quad 
V''(p)\ge \frac{A-\pi^2/4}{p}.
\]
Since
$
A=\frac{2G}{\ln2}\approx 2.6434
$ and $
\frac{\pi^2}{4}\approx 2.4674,
$
we have $A>\pi^2/4$, and therefore
\[
V''(p)>0
\qquad\text{for all } p\in\Bigl(0,\frac12\Bigr].
\]
So $V$ is convex on $(0,1/2]$.

Taking $ p\downarrow0$ and using \eqref{eq:S(p)-approx} we get
$
V(p)\to 0
$ as $ (p\downarrow0).
$
Moreover,
\[
V\Bigl(\frac12\Bigr)
=
S\Bigl(\frac12\Bigr)-A\cdot \frac12\log2
=
G-\frac{2G}{\ln2}\cdot \frac12\ln2
=
0.
\]
Thus the continuous extension of $V$ to $[0,1/2]$ is convex and vanishes at both endpoints. Hence 
\[
V(p)\le 0
\qquad \text{for all } p\in\Bigl[0,\frac12\Bigr].
\]
Therefore
by symmetry,
\[
S(p)\le A\,p^*\log\frac1{p^*},
\qquad 0\le p\le 1.
\]
Multiplying by $\frac{2\sqrt2}{\pi^{3/2}}$ and using again
$
\log\frac1{p^*}=(\ln2)\log_2\frac1{p^*},
$
gives
\[
\tilde U(p,0)
\le
\frac{4\sqrt2\,G}{\pi^{3/2}}\,p^*\log_2\frac1{p^*}.
\]

\textbf{Step 3: Sharpness.}
The sharpness of the lower constant follows from the asymptotic already proved in Lemma~\ref{lem:Utilde-asymptotic-S2}:
\[
\tilde U(p,0)
=
\sqrt{\frac2\pi}\,p\log\frac1p+O(p)
\qquad (p\downarrow0).
\]
Dividing by $p\log_2(1/p)=\frac{1}{\ln2}p\log(1/p)$ yields
\[
\lim_{p\downarrow0}
\frac{\tilde U(p,0)}{p\log_2\frac1p}
=
\sqrt{\frac2\pi}\,\ln2.
\]
By symmetry, the same limit holds as $p\uparrow1$, so the lower constant cannot be improved.

For the upper bound, equality holds at $p=\frac12$, because
\[
S\Bigl(\frac12\Bigr)=G
\]
and therefore
\[
\tilde U\Bigl(\frac12,0\Bigr)
=
\frac{2\sqrt2}{\pi^{3/2}}\,G
=
\frac{4\sqrt2\,G}{\pi^{3/2}}\cdot \frac12
=
\frac{4\sqrt2\,G}{\pi^{3/2}}\,
\Bigl(\frac12\Bigr)\log_2 2.
\]
This proves sharpness of the upper constant.
\end{proof}

We now introduce the extremal profile in the Borel class:
\[
\mathcal{F}(p):=\inf\Bigl\{\|S_2(\mathbbm 1_A)\|_1:\ A\subset{[0,1)}\ \text{Borel},\ |A|=p\Bigr\},
\qquad p\in[0,1].
\]

\begin{lemma}\label{lem:F-continuity-S2}
For all $p,q\in[0,1]$,
\[
|\mathcal{F}(p)-\mathcal{F}(q)|\le |p-q|^{1/2}.
\]
In particular, $\mathcal{F}$ is continuous on $[0,1]$.
\end{lemma}

\begin{proof}
Fix $p,q\in[0,1]$ and $\varepsilon>0$. Choose a Borel set $A\subset{[0,1)}$ with $|A|=p$ and
\[
\|S_2(\mathbbm 1_A)\|_1\le \mathcal{F}(p)+\varepsilon.
\]
Assume first that $q\ge p$. Since Lebesgue measure is nonatomic, there exists a Borel set $E\subset{[0,1)}\setminus A$ with $|E|=q-p$. Set $B:=A\cup E$. Then $|B|=q$ and
$
|A\triangle B|=q-p.
$
So,
\[
\big|\|S_2(\mathbbm 1_A)\|_1-\|S_2(\mathbbm 1_B)\|_1\big|
\le |A\triangle B|^{1/2}
=
|p-q|^{1/2}.
\]
Hence
\[
\mathcal{F}(q)\le \|S_2(\mathbbm 1_B)\|_1\le \mathcal{F}(p)+\varepsilon+|p-q|^{1/2}.
\]
Letting $\varepsilon\downarrow0$ gives
\[
\mathcal{F}(q)\le \mathcal{F}(p)+|p-q|^{1/2}.
\]
Interchanging the roles of $p$ and $q$ proves the claim.
\end{proof}

It remains to prove the matching upper bound.
\begin{lemma}\label{lem:dyadic-mass-upper-S2}
Let $p_m=k_m/2^m\in(0,1)$ and set $a_m=2^{-m}$. Then there exists a Borel set
$A_m^\infty\subset{[0,1)}$ with $|A_m^\infty|=p_m$ such that
\[
\|S_2(\mathbbm 1_{A_m^\infty})\|_1
=
a_m\,\mathbb E\sqrt{\tau_m},
\]
where $\tau_m$ is the exit time of the dyadic lattice walk started at $p_m$. 
Moreover, if $p_m\to p\in(0,1)$, then
\[
a_m\,\mathbb E\sqrt{\tau_m}\to \mathbb E_p\sqrt{\tau}.
\]
\end{lemma}

\begin{proof}
Let $(r_k)_{k\ge1}$ be the Rademacher system on {$[0,1)$}, give by
\[
  r_k(x) := \operatorname{sgn}\bigl(\sin(2^k \pi x)\bigr),
  \qquad x\in{[0,1)}.
\]
Define
\[
X_0^{(m)}:=p_m,\qquad
X_n^{(m)}:=X_{n-1}^{(m)}+a_m r_n,
\qquad n\ge1.
\]
Let
\[
\tau_m:=\inf\{n\ge0:\ X_n^{(m)}\in\{0,1\}\},
\qquad
A_m^\infty:=\{X_{\tau_m}^{(m)}=1\}.
\]
Since
\[
A_m^\infty=\bigcup_{n\ge0}\bigl(\{\tau_m=n\}\cap\{X_n^{(m)}=1\}\bigr),
\]
the set $A_m^\infty$ is Borel.
The stopped process $X_{n\wedge\tau_m}^{(m)}$ is a bounded martingale with values in $[0,1]$. Since
\[
X_{\tau_m}^{(m)}\in\{0,1\}
\implies
X_{\tau_m}^{(m)}=\mathbbm 1_{A_m^\infty}.
\]
Hence optional stopping gives
\[
|A_m^\infty|
=
\mathbb E[X_{\tau_m}^{(m)}]
=
X_0^{(m)}
=
p_m.
\]
Again by optional stopping,
\[
\mathbb E(\mathbbm 1_{A_m^\infty}\mid\mathcal F_n)=X_{n\wedge\tau_m}^{(m)}.
\]
Hence
\[
d_n
=
X_{n\wedge\tau_m}^{(m)}-X_{(n-1)\wedge\tau_m}^{(m)}
=
a_m r_n\,\mathbbm{1}_{\{n\le\tau_m\}}.
\]
Therefore
\[
S_2(\mathbbm 1_{A_m^\infty})^2
=
\sum_{n\ge1}|d_n|^2
=
a_m^2\tau_m,
\]
and so
\[
\|S_2(\mathbbm 1_{A_m^\infty})\|_1
=
a_m\,\mathbb E\sqrt{\tau_m}.
\]

To prove convergence, let $Y_m$ be the continuous linear interpolation of the walk $(X_n^{(m)})_{n\ge0}$ on the time scale $a_m^2$, and define
\[
\theta_m:=\inf\{t\ge0:\, Y_m(t)\notin(0,1)\}.
\]
Then $\theta_m=a_m^2\tau_m$. By Donsker's invariance principle, the processes $Y_m$ converge weakly to Brownian motion started at $p$ \cite{Durrett2019}. By Skorokhod representation \cite{Billingsley1999}, we may see them on a common probability space so that
\[
Y_m\to B
\qquad\text{almost surely uniformly on compact time intervals}
\]

We claim that
\[
\theta_m\to\tau
\qquad\text{almost surely.}
\]
Fix a sample point $\omega$ for which the above uniform convergence holds and such that Brownian motion, after hitting $0$ or $1$, leaves the interval $[0,1]$ immediately; this is a standard almost sure property of Brownian motion. Let $\varepsilon>0$. By continuity of $B(\cdot,\omega)$ on $[0,\tau(\omega)-\varepsilon]$, there exists $\delta>0$ such that
\[
B_t(\omega)\in[\delta,1-\delta]
\qquad\text{for all }0\le t\le \tau(\omega)-\varepsilon.
\]
Hence for all sufficiently large $m$,
\[
Y_m(t,\omega)\in(0,1)
\qquad\text{for all }0\le t\le \tau(\omega)-\varepsilon,
\]
so
\[
\theta_m(\omega)\ge \tau(\omega)-\varepsilon.
\]
On the other hand, if $B_{\tau(\omega)}(\omega)=0$, then by the above boundary regularity there exists
$t_\varepsilon\in[\tau(\omega),\tau(\omega)+\varepsilon]$ such that
\[
B_{t_\varepsilon}(\omega)<0.
\]
Similarly, if $B_{\tau(\omega)}(\omega)=1$, there exists
$t_\varepsilon\in[\tau(\omega),\tau(\omega)+\varepsilon]$ such that
\[
B_{t_\varepsilon}(\omega)>1.
\]
Uniform convergence on $[0,t_\varepsilon]$ then implies that for all sufficiently large $m$,
\[
Y_m(t_\varepsilon,\omega)\notin(0,1),
\]
and therefore
\[
\theta_m(\omega)\le t_\varepsilon\le \tau(\omega)+\varepsilon.
\]
Since $\varepsilon>0$ is arbitrary, it follows that $\theta_m(\omega)\to\tau(\omega)$.
On the other hand, the process
\[
X_{n\wedge\tau_m}^{(m)}\bigl(1-X_{n\wedge\tau_m}^{(m)}\bigr)+a_m^2(n\wedge\tau_m)
\]
is a martingale, and therefore
\[
\mathbb E[\theta_m]=a_m^2\mathbb E[\tau_m]=p_m(1-p_m)\le \frac14.
\]
Thus $\{\sqrt{\theta_m}\}$ is bounded in $L^2$, hence uniformly integrable. Consequently,
\[
a_m\,\mathbb E\sqrt{\tau_m}
=
\mathbb E\sqrt{\theta_m}
\to
\mathbb E_p\sqrt{\tau}.
\]
\end{proof}

\begin{proof}[Proof of Theorem~\ref{mth01}]
By Proposition~\ref{prop:borel-extension-S2},
\[
\mathcal{F}(p)\ge \tilde U(p,0)=\mathbb E_p\sqrt{\tau}
\qquad\text{for all }p\in[0,1].
\]

Fix $p\in(0,1)$, and choose dyadic numbers $p_m=k_m/2^m\to p$. By Lemma~\ref{lem:dyadic-mass-upper-S2},
\[
\mathcal{F}(p_m)\le \|S_2(\mathbbm 1_{A_m^\infty})\|_1
=
a_m\,\mathbb E\sqrt{\tau_m}
\to
\mathbb E_p\sqrt{\tau}.
\]
On the other hand, Proposition~\ref{prop:borel-extension-S2} gives
\[
\mathcal{F}(p_m)\ge \mathbb E_{p_m}\sqrt{\tau},
\]
and Lemma~\ref{lem:Utilde-asymptotic-S2} implies
$\mathbb E_{p_m}\sqrt{\tau}\to \mathbb E_p\sqrt{\tau}.$
Hence
$
\mathcal{F}(p_m)\to \mathbb E_p\sqrt{\tau}.
$
Since $\mathcal{F}$ is continuous by Lemma~\ref{lem:F-continuity-S2},
\[
\mathcal{F}(p)=\lim_{m\to\infty}\mathcal{F}(p_m)=\mathbb E_p\sqrt{\tau}.
\]
The cases $p=0$ and $p=1$ are immediate, since both sides vanish. Therefore
\[
\inf\Bigl\{\|S_2(\mathbbm 1_A)\|_1:\ A\subset{[0,1)}\ \text{Borel},\ |A|=p\Bigr\}
=
\mathbb E_p\sqrt{\tau}.
\]
And so from Lemma~\ref{lem:Utilde-sharp-constants-S2} the two-sided estimate follows 
\[
\sqrt{\frac2\pi}\,\ln 2 \; p^*\log_2\frac1{p^*}
\;\le\;
\mathbb E_p\sqrt{\tau}
\;\le\;
\frac{4\sqrt2\,G}{\pi^{3/2}}\,p^*\log_2\frac1{p^*}.
\]

\end{proof}

\section{Proof of Theorem \ref{mth03}}

\subsection{Lower bound}
For all $\alpha \in (0,1)$ and all Borel sets $A\subset {[0,1)}$, we want to prove
$\|S_1(\mathbbm{1}_{A})\|_{\alpha} \geq |A|^*$. By Remark~\ref{supersol}, it is enough to find a continuous function $B:[0,1]\to[0,\infty)$ such that $B(0)=B(1)=0$ and $B$ satisfies the two-point inequality \eqref{twopp} for $(\alpha,\beta)=(\alpha,1)$. Then $\|S_1(\mathbbm 1_A)\|_\alpha \ge B(|A|)$ for every Borel set $A\subset{[0,1)}$. We take $B(x)=2x(1-x)$, so it only remains to verify the two-point inequality.

\begin{lemma}\label{teclem}
Let $B(x)=2x(1-x)$. Then $B$ satisfies the two-point inequality \eqref{twopp} with $(\alpha,\beta)=(\alpha,1)$, $\alpha \in (0,1)$, that is,
\begin{align*}
&\left( 2x(1-x)+\left|\frac{x-y}{2}\right|\right)^{\alpha}
+
\left( 2y(1-y)+\left|\frac{x-y}{2}\right|\right)^{\alpha}\\
&\hspace{3cm}\geq 2\left( (x+y)\left(1-\frac{x+y}{2}\right)\right)^{\alpha}
\end{align*}
for all $x,y \in [0,1]$.
\end{lemma}

\begin{proof}
Let $x,y\in[0,1]$. The cases $x=y=0$ and $x=y=1$ are immediate, so we may assume that at least one of the two quantities $2x(1-x)+|x-y|/2$ and $2y(1-y)+|x-y|/2$ is positive. Define
\begin{equation}\label{eq:falpha}
f(\alpha):=
\left[\frac{1}{2} \left( 2x(1-x)+\frac{|x-y|}{2}\right)^{\alpha} + \frac{1}{2}\left( 2y(1-y)+\frac{|x-y|}{2}\right)^{\alpha}\right]^{\frac{1}{\alpha}}.
\end{equation}
Consider the Bernoulli random variable $X$ given by
\[
X=
\begin{cases}
   2x(1-x)+\frac{|x-y|}{2} & \text{with probability } 1/2,\\[1mm]
   2y(1-y)+\frac{|x-y|}{2} & \text{with probability } 1/2.
\end{cases}
\]
Then $f(\alpha)= \left( \E(X^{\alpha})\right)^{1/\alpha}$. Since $X\ge 0$, the monotonicity of power means implies that $\alpha\mapsto f(\alpha)$ is nondecreasing on $(0,\infty)$, so $f(\alpha)\ge \lim_{t\to 0^+} f(t) = e^{\E\ln X}$ in the nontrivial cases under consideration.

Thus, to verify the inequality in Lemma~\ref{teclem}, it suffices to show that 
\begin{align*}
\E \ln X
&=
\frac{1}{2} \ln\left( 2x(1-x)+\frac{|x-y|}{2}\right)
+
\frac{1}{2}\ln \left( 2y(1-y)+\frac{|x-y|}{2}\right)\\
&\geq \ln \left( 2\frac{x+y}{2}\left( 1-\frac{x+y}{2}\right)\right).
\end{align*}
This inequality is equivalent to
\[
\left( 2x(1-x)+\frac{|x-y|}{2}\right)\left( 2y(1-y)+\frac{|x-y|}{2}\right)
\geq 4 \left( \frac{x+y}{2}\left( 1-\frac{x+y}{2}\right)\right)^2 .
\]
Without loss of generality, assume $x>y$. Expanding both sides gives
\[
-\frac{1}{4}\left(  4x^2-5x+y\right)(x+3y-4y^{2})\geq \frac{1}{4}(x+y-2)^2(x+y)^2,
\]
and moving all terms to the right yields $(x-y)(x-y+1)(x^2+x(6y-1)+(y-7)y)\leq 0$. It therefore suffices to show that $P(x):=x^2+x(6y-1)+(y-7)y\leq 0$.
Notice that $P(x)$ is convex in $x$, hence it is enough to check that $P$ is non-positive at the endpoints of the interval $[y,1]$.
We have
\begin{align*}
P(y) &= y^2+y(6y-1)+(y-7)y = 8y(y-1)\leq 0,\\
P(1) &= 1+6y-1+(y-7)y = y(y-1)\leq 0.
\end{align*}
This finishes the proof of Lemma~\ref{teclem}. 
\end{proof}

Now let $A\subset{[0,1)}$ be a Borel set. By Remark~\ref{supersol} in the continuous case and Lemma~\ref{teclem}, we obtain $\|S_1(\mathbbm{1}_{A})\|_{\alpha} \geq B(|A|)=2|A|(1-|A|)$. Since $2x(1-x)\geq \min\{x,1-x\}=x^{*}$ for all $x\in[0,1]$, it follows that $\|S_1(\mathbbm{1}_{A})\|_{\alpha} \geq |A|^*$. This finishes the proof of the lower bound in Theorem~\ref{mth03}.

\subsection{Sharpness}

By Theorem~\ref{mth02}, the case $\alpha=1$ is sharp. We now prove 
sharpness in Theorem~\ref{mth03} for every $\alpha\in(0,1)$.
For each $m\ge 1$, let
\[
A_m:=[0,2^{-m}).
\]
Then $A_m\in\mathcal D$ and
$|A_m|=2^{-m}\downarrow 0.$
We will show that there exists a constant $C_\alpha<\infty$ such that
\[
\|S_1(\mathbbm{1}_{A_m})\|_\alpha \le C_\alpha |A_m|
\qquad \text{for all } m\ge 1.
\]
This would prove the desired sharpness along the sequence $p_m=2^{-m}$.
Set
\[
f:=\mathbbm{1}_{A_m},
\qquad
f_n:=\mathbb E(f\mid \mathcal F_n),
\qquad
d_n:=f_n-f_{n-1},
\]
with $f_0=\mathbb E f = 2^{-m}$.
For $1\le n\le m$, the function $f_n$ is constant on dyadic intervals of length
$2^{-n}$, and since $A_m=[0,2^{-m})\subset[0,2^{-n})$, we have
\[
f_n=2^{\,n-m}\,\mathbbm{1}_{[0,2^{-n})}.
\]
For $n\ge m$, we have $f_n=f$, so in particular $d_n=0$ for $n\ge m+1$.
Now fix $1\le n\le m$. Since
\[
f_n=2^{\,n-m}\,\mathbbm{1}_{[0,2^{-n})}
\quad\text{and}\quad
f_{n-1}=2^{\,n-1-m}\,\mathbbm{1}_{[0,2^{-n+1})},
\]
for $ 1\le n\le m$ we obtain
\[
d_n
=
2^{\,n-1-m}
\Bigl(
\mathbbm{1}_{[0,2^{-n})}
-
\mathbbm{1}_{[2^{-n},\,2^{-n+1})}
\Bigr)
\implies
|d_n|
=
2^{\,n-1-m}\,\mathbbm{1}_{[0,\,2^{-n+1})},
\]
and hence
\[
S_1(\mathbbm{1}_{A_m})(t)
=
\sum_{n=1}^\infty |d_n(t)|
=
\sum_{n=1}^m 2^{\,n-1-m}\,\mathbbm{1}_{[0,\,2^{-n+1})}(t).
\]
We now compute this function explicitly. Let $1\le k\le m$ and suppose that
$t\in [2^{-k},\,2^{-k+1}).$
Then
\[
t<2^{-n+1}
\quad\Longleftrightarrow\quad
n\le k,
\]
so exactly the terms $n=1,\dots,k$ contribute to the sum. Thus
\[
S_1(\mathbbm{1}_{A_m})(t)
=
\sum_{n=1}^k 2^{\,n-1-m}
=
2^{-m}\sum_{n=1}^k 2^{n-1}
=
2^{-m}(2^k-1).
\]
If $t\in[0,2^{-m})$, then every term $n=1,\dots,m$
contributes, and so
\[
S_1(\mathbbm{1}_{A_m})(t)
=
\sum_{n=1}^m 2^{\,n-1-m}
=
2^{-m}(2^m-1)
=
1-2^{-m}.
\]
Since
$
[0,2^{-m}),\ [2^{-m},2^{-m+1}),\ \dots,\ {[2^{-1},1)}
$
form a partition of {$[0,1)$}, we get
\begin{align*}
\int_0^1 S_1(\mathbbm{1}_{A_m})^\alpha
&=
\sum_{k=1}^m
\int_{2^{-k}}^{2^{-k+1}}
\bigl(2^{-m}(2^k-1)\bigr)^\alpha\,dt
+
\int_0^{2^{-m}} (1-2^{-m})^\alpha\,dt \\
&=
\sum_{k=1}^m
2^{-k}\bigl(2^{-m}(2^k-1)\bigr)^\alpha
+
2^{-m}(1-2^{-m})^\alpha.
\end{align*}
Using $(2^k-1)^\alpha\le 2^{\alpha k}$ and $(1-2^{-m})^\alpha\le 1$, we obtain
\begin{align*}
\int_0^1 S_1(\mathbbm{1}_{A_m})^\alpha
&\le
2^{-\alpha m}\sum_{k=1}^m 2^{-k}2^{\alpha k}
+
2^{-m} =
2^{-\alpha m}\sum_{k=1}^m 2^{-(1-\alpha)k}
+
2^{-m}.
\end{align*}
Since $0<\alpha<1$, the geometric series
$
\sum_{k=1}^\infty 2^{-(1-\alpha)k}
$
converges. Also, because $\alpha<1$, we have $2^{-m}\le 2^{-\alpha m}$ for all
$m\ge1$. Therefore there exists a constant $C_\alpha<\infty$ such that
\[
\int_0^1 S_1(\mathbbm{1}_{A_m})^\alpha
\le
C_\alpha\,2^{-\alpha m}
=
C_\alpha\,|A_m|^\alpha.
\]
Thus, with $p_m:=|A_m|=2^{-m}\downarrow0$ and taking the power $1/\alpha$, we have constructed a sequence $A_m\in\mathcal D$ such that 
\[
\|S_1(\mathbbm{1}_{A_m})\|_\alpha
\le
C_\alpha^{1/\alpha}\,2^{-m}
=
C_\alpha^{1/\alpha}\,|A_m| = 
C_\alpha^{1/\alpha}\, p_m
.
\]
which proves the sharpness statement in Theorem~\ref{mth03}.
\qedhere

\section{Proof of Proposition \ref{twopp-max-func}}

Throughout this section assume
$
0<\alpha\le 1\le \beta$ and $
r:=\frac{\alpha}{\beta}\in(0,1].
$
Let $\mathcal F_{\alpha,\beta}$ be the family of all $f\in C([0,1])$ such that
$f\ge0$, $f(0)=f(1)=0$, and
\[
f\!\left(\frac{x+y}{2}\right)^\alpha
\le
\frac12\Bigl(f(x)^\beta+d^\beta\Bigr)^r
+
\frac12\Bigl(f(y)^\beta+d^\beta\Bigr)^r,
\qquad
d:=\left|\frac{x-y}{2}\right|,
\]
for all $x,y\in[0,1]$. Define
\[
B_{\alpha,\beta}(t):=\sup\{f(t):f\in\mathcal F_{\alpha,\beta}\},
\qquad t\in[0,1].
\]
We show that $B_{\alpha,\beta}$ is finite, continuous, and belongs to $\mathcal F_{\alpha,\beta}$.

Let $\phi(t):=\operatorname{dist}(t,\mathbb Z)$ and
\[
T_\alpha(t):=\sum_{n=0}^\infty 2^{-n}\phi(2^n t)^\alpha,
\qquad t\in[0,1].
\]

\begin{lemma}\label{lem:Takagi-growth}
The series defining $T_\alpha$ converges uniformly on $[0,1]$. Hence
$T_\alpha\in C([0,1])$, $T_\alpha(0)=T_\alpha(1)=0$, and there exists
$K_\alpha<\infty$ such that
\[
T_\alpha(t)\le
\begin{cases}
K_\alpha\, t^\alpha, & 0<\alpha<1,\\[2mm]
K_1\, t\log\!\bigl(\dfrac{2}{t}\bigr), & \alpha=1,
\end{cases}
\qquad t\in(0,1].
\]
\end{lemma}

\begin{proof}
Since $0\le \phi\le 1/2$, we have
$0\le 2^{-n}\phi(2^n t)^\alpha\le 2^{-n-\alpha}$ for all $t\in[0,1]$, so the
series converges uniformly; continuity and the endpoint values follow.

Fix $t\in(0,1]$ and let $N$ be maximal with $2^N t\le 1/2$ (possibly $N=-1$).
Then
\[
T_\alpha(t)
=
\sum_{n=0}^{N}2^{-n}\phi(2^n t)^\alpha
+
\sum_{n=N+1}^{\infty}2^{-n}\phi(2^n t)^\alpha.
\]
For $n\le N$, $\phi(2^n t)\le 2^n t$, so
\[
\sum_{n=0}^{N}2^{-n}\phi(2^n t)^\alpha
\le
t^\alpha\sum_{n=0}^{N}2^{(\alpha-1)n}.
\]
For $n\ge N+1$, $\phi(2^n t)\le 1/2$, hence
\[
\sum_{n=N+1}^{\infty}2^{-n}\phi(2^n t)^\alpha
\le
2^{-\alpha}\sum_{n=N+1}^{\infty}2^{-n}
=
2^{-\alpha}2^{-N}.
\]
Since $2^{N+1}t>1/2$, we have $2^{-N}<4t$. If $0<\alpha<1$, then
$t\le t^\alpha$ on $(0,1]$, so both pieces are $O(t^\alpha)$. If $\alpha=1$,
the first sum is $(N+1)t$ and $(N+1)\lesssim \log(2/t)$, while the tail is
$O(t)\le O(t\log(2/t))$. This proves the claim.
\end{proof}

\begin{lemma}\label{lem:reduce-midconvex}
Let $f\in\mathcal F_{\alpha,\beta}$ and set $A:=f^\alpha$. Then
\begin{equation}\label{eq:approx-midconvex}
A\!\left(\frac{x+y}{2}\right)
\le
\frac{A(x)+A(y)}{2}
+
\left|\frac{x-y}{2}\right|^\alpha,
\qquad x,y\in[0,1].
\end{equation}
\end{lemma}

\begin{proof}
Let $d:=|(x-y)/2|$. Since $r\in(0,1]$, $(u+v)^r\le u^r+v^r$ for all $u,v\ge0$.
Applying this with $u=f(x)^\beta$ and $v=d^\beta$, and similarly with $x$
replaced by $y$, gives
\[
\bigl(f(x)^\beta+d^\beta\bigr)^r\le f(x)^\alpha+d^\alpha=A(x)+d^\alpha,
\]
\[
\bigl(f(y)^\beta+d^\beta\bigr)^r\le f(y)^\alpha+d^\alpha=A(y)+d^\alpha.
\]
Substituting into the two-point inequality yields \eqref{eq:approx-midconvex}.
\end{proof}

Define the doubling map $D:[0,1]\to[0,1]$ by
\[
D(t)=
\begin{cases}
2t, & 0\le t<1/2,\\
2t-1, & 1/2\le t\le 1.
\end{cases}
\]
Then $D(0)=0$, $D(1)=1$, and for every $t\in[0,1]$,
\[
A(t)\le \frac12 A(D(t)) + C\,\phi(t)^\alpha
\]
whenever $A(0)=A(1)=0$ and \eqref{eq:approx-midconvex} holds with constant $C$.

\begin{lemma}\label{lem:Takagi-bound-A}
Let $A\in C([0,1])$ be nonnegative, with $A(0)=A(1)=0$, and assume
\[
A\!\left(\frac{x+y}{2}\right)
\le
\frac{A(x)+A(y)}{2}
+
C\left|\frac{x-y}{2}\right|^\alpha,
\qquad x,y\in[0,1].
\]
Then
\begin{equation}\label{eq:Takagi-bound}
A(t)\le C\,T_\alpha(t),
\qquad t\in[0,1].
\end{equation}
\end{lemma}

\begin{proof}
Fix $t\in[0,1]$. Applying the midpoint inequality to $(0,D(t))$ if $t\le1/2$
and to $(D(t),1)$ if $t\ge1/2$ gives
\[
A(t)\le \tfrac12 A(D(t))+C\,\phi(t)^\alpha.
\]
Iterating, for every $N\ge1$,
\[
A(t)\le 2^{-N}A(D^N(t))+C\sum_{n=0}^{N-1}2^{-n}\phi(D^n(t))^\alpha.
\]
Since $A$ is bounded on $[0,1]$, the first term tends to $0$ as $N\to\infty$.
Also $\phi(D^n(t))=\phi(2^n t)$ for all $n\ge0$, because $D^n(t)\equiv 2^n t
\pmod 1$ and $\phi$ is the distance to the integers. Letting $N\to\infty$
yields \eqref{eq:Takagi-bound}.
\end{proof}

\begin{lemma}\label{lem:chord-est}
Let $A$ satisfy the assumptions of Lemma~\ref{lem:Takagi-bound-A}.
Then for every $0\le a<b\le1$ and $t\in[0,1]$,
\begin{equation}\label{eq:chord-est}
A((1-t)a+tb)
\le
(1-t)A(a)+tA(b)+(b-a)^\alpha T_\alpha(t).
\end{equation}
\end{lemma}

\begin{proof}
Define
\[
H(t):=A((1-t)a+tb)-\bigl((1-t)A(a)+tA(b)\bigr).
\]
Then $H(0)=H(1)=0$. If $u=(1-s)a+sb$ and $v=(1-t)a+tb$, then
\[
\tfrac{u+v}{2}=\Bigl(1-\tfrac{s+t}{2}\Bigr)a+\tfrac{s+t}{2}b,
\qquad
\left|\tfrac{u-v}{2}\right|=(b-a)\left|\tfrac{s-t}{2}\right|.
\]
Applying the midpoint inequality for $A$ gives
\[
H\!\left(\frac{s+t}{2}\right)
\le
\frac{H(s)+H(t)}{2}
+
(b-a)^\alpha\left|\frac{s-t}{2}\right|^\alpha.
\]
Let $H_+:=\max(H,0)$. Since $u\mapsto u_+$ is convex and
$(m+E)_+\le m_+ + E$ for $E\ge0$, we obtain
\[
H_+\!\left(\frac{s+t}{2}\right)
\le
\frac{H_+(s)+H_+(t)}{2}
+
(b-a)^\alpha\left|\frac{s-t}{2}\right|^\alpha.
\]
Applying Lemma~\ref{lem:Takagi-bound-A} to $H_+$ yields
\[
H_+(t)\le (b-a)^\alpha T_\alpha(t).
\]
Since $H\le H_+$, this proves \eqref{eq:chord-est}.
\end{proof}

\begin{proof}[Proof of Proposition~\ref{twopp-max-func}]
We divide the proof into three steps.

\smallskip
\noindent\textit{Step 1: $B_{\alpha,\beta}$ is finite.}
Fix $f\in\mathcal F_{\alpha,\beta}$ and set $A:=f^\alpha$. By
Lemma~\ref{lem:reduce-midconvex}, $A$ satisfies \eqref{eq:approx-midconvex}
with $A(0)=A(1)=0$, so Lemma~\ref{lem:Takagi-bound-A} gives
\[
A(t)\le T_\alpha(t),\qquad t\in[0,1].
\]
Hence $f(t)\le T_\alpha(t)^{1/\alpha}$ for all $t$, and therefore
$B_{\alpha,\beta}(t)<\infty$.

\smallskip
\noindent\textit{Step 2: $B_{\alpha,\beta}$ is continuous.}
Let $\mathcal A:=\{f^\alpha:f\in\mathcal F_{\alpha,\beta}\}$. By
Lemma~\ref{lem:reduce-midconvex}, each $A\in\mathcal A$ satisfies
\eqref{eq:approx-midconvex}. Fix $0\le x<y\le1$ and set $\delta:=y-x$.
Applying Lemma~\ref{lem:chord-est} on $[x,1]$ with $t=\delta/(1-x)$ and using
$A(1)=0$ gives
\[
A(y)-A(x)\le (1-x)^\alpha T_\alpha\!\left(\frac{\delta}{1-x}\right).
\]
Applying Lemma~\ref{lem:chord-est} on $[0,y]$ with $t=x/y$ and using $A(0)=0$
gives
\[
A(x)-A(y)\le y^\alpha T_\alpha\!\left(\frac{\delta}{y}\right).
\]
By Lemma~\ref{lem:Takagi-growth}, both right-hand sides are bounded by
\[
\omega(\delta):=
\begin{cases}
K_\alpha\,\delta^\alpha, & 0<\alpha<1,\\[2mm]
K_1\,\delta\log\!\bigl(\dfrac{2}{\delta}\bigr), & \alpha=1,\ 0<\delta\le1,\\[2mm]
0, & \delta=0.
\end{cases}
\]
Hence every $A\in\mathcal A$ satisfies
\[
|A(x)-A(y)|\le \omega(|x-y|),
\qquad x,y\in[0,1].
\]
Therefore
$
A_{\alpha,\beta}(t):=\sup_{A\in\mathcal A}A(t)
$
is continuous, since
\[
|A_{\alpha,\beta}(x)-A_{\alpha,\beta}(y)|
\le
\sup_{A\in\mathcal A}|A(x)-A(y)|
\le
\omega(|x-y|).
\]
Finally, we get that $B_{\alpha,\beta}=A_{\alpha,\beta}^{1/\alpha}$ is continuous from
\[
A_{\alpha,\beta}(t)
=
\sup_{f\in\mathcal F_{\alpha,\beta}} f(t)^\alpha
=
\left(\sup_{f\in\mathcal F_{\alpha,\beta}}f(t)\right)^\alpha
=
B_{\alpha,\beta}(t)^\alpha,
\]

\smallskip
\noindent\textit{Step 3: $B_{\alpha,\beta}\in\mathcal F_{\alpha,\beta}$.}
Nonnegativity and the boundary conditions are inherited from the supremum. Fix
$x,y\in[0,1]$ and write $d:=|(x-y)/2|$. Since
\[
F_d(s):=(s^\beta+d^\beta)^r
\]
is increasing on $[0,\infty)$, every $f\in\mathcal F_{\alpha,\beta}$ satisfies
\[
f\!\left(\frac{x+y}{2}\right)^\alpha
\le
\frac12 F_d(f(x))+\frac12 F_d(f(y))
\le
\frac12 F_d(B_{\alpha,\beta}(x))+\frac12 F_d(B_{\alpha,\beta}(y)).
\]
Taking the supremum over $f$ yields
\[
B_{\alpha,\beta}\!\left(\frac{x+y}{2}\right)^\alpha
\le
\frac12\Bigl(B_{\alpha,\beta}(x)^\beta+d^\beta\Bigr)^r
+
\frac12\Bigl(B_{\alpha,\beta}(y)^\beta+d^\beta\Bigr)^r.
\]
Thus $B_{\alpha,\beta}\in\mathcal F_{\alpha,\beta}$, and pointwise maximality is
built into the definition.
\end{proof}

\newpage

\appendix
\section{Appendix: Proof of Theorem \ref{thm-bn}}
In this appendix we give a proof of the main
properties of $B_{1,1}$ in Theorem \ref{thm-bn}, with some results from \cite{hart} and \cite{BIM} for bounding purposes.
The point of this argument is to construct $B_{1,1}$ explicitly on dyadic rationals through the approximants $B_n$, identify the limit $P$ as the maximal function satisfying the two-point inequality in the case $(\alpha,\beta)=(1,1)$, and derive its basic structural properties. In particular, this gives a direct proof of the symmetry, scaling relation, lower bound, and dyadic modulus of continuity that were announced in the main text.

\begin{lemma}\label{lm:F1}
    For any integer $k$,  $0\leq k\leq 2^n$, we have 
    \[
    F(2k)=2F(k)+k,
    \]
    where $F$ is defined as in \eqref{eq:F_def}. 
\end{lemma}
\begin{proof}
    From {\cite[Remark 2.5]{hart}}, we know that $F$ satisfies $F(1)=0$, and 
    \begin{equation}
    \label{eq:Fmax}
        F(l) = \max_{0\leq m\leq l/2} (F(m)+F(l-m)+m),
    \end{equation}
    where the maximum in the right hand side of \eqref{eq:Fmax} is achieved at $m=\lfloor l/2\rfloor$. Choosing $l=2k$ proves the lemma. 

\end{proof}

\begin{lemma}\label{lm:F2}
For nonnegative integers $a$ and $b$, it holds
\[
F(a+b) -\min(a,b) \geq F(a) + F(b) .
\]
\end{lemma}

\begin{proof}
Assume without loss of generality that $a\le b$, so $\min(a,b)=a$.
Then $2a\le a+b$, hence $a\le \lfloor (a+b)/2\rfloor$.
Applying {\cite[Proposition 2.1]{hart}} with $\ell=a+b$ and choosing $m=a$, we get
\[
F(a+b)=\max_{0\le m\le \lfloor (a+b)/2\rfloor}\bigl(F(m)+F(a+b-m)+m\bigr)
\ge F(a)+F(b)+a.
\]
Rearranging yields $F(a+b)-a\ge F(a)+F(b)$, which is the claim.
\end{proof}

To prove (i), we show that for all $x \in D= \cup_{n \geq 0} D_{n}$ the limit  $\lim_{n \to \infty}B_{n}(x)$, where $B_n$ is defined in \eqref{eq:bn}, exists and is denoted by $P(x)$. 
i.e.
\[
\lim_{n\to\infty} B_n(x) 
= \lim_{n\to\infty} 
\left(   nx - \frac{1}{2^{n-1}} F(x2^n)
   \right) = P(x).
\]
Recall that $B_n$ is defined on ${D}_n$, and $D_n\subseteq D_{n+1}$. It suffices to show that, $B_{n+1}\mid_{D_n}= B_n$  for all $n\geq 1$.
Let $x\in D_n$, then $x=\frac{k}{2^n}$ for some $0\leq k \leq 2^n$.
Using Lemma \ref{lm:F1}, we have
\begin{align*}
    B_{n+1}(x) 
    &=  (n+1)\frac{k}{2^n} - \frac{1}{2^n}F\left( 2k \right)= (n+1)\frac{k}{2^n} - \frac{1}{2^n}\left( k + 2F(k) \right)\\
    & = \frac{n}{2^n}k + \frac{k}{2^n}-\frac{k}{2^n} - \frac{1}{2^{n-1}} F(k)= B_n(x).
\end{align*}
Hence the limit exists.\\

To prove (ii), we first show that the limit function $P: D \mapsto [0,\infty)$ defined on $D$ indeed satisfies $P(0)=P(1)=0$, and the two-point inequality \eqref{twopp} with $(\alpha, \beta)=(1,1)$, that is 
\begin{align}\label{ori}
\left| \frac{x-y}{2}\right| + \frac{P(x)+P(y)}{2} \geq P\left( \frac{x+y}{2} \right) \quad \text{for all} \quad x,y \in D. 
\end{align}
Note that indeed $P(0)=B_{1}(0)=0$, and $P(1)=B_{1}(1)=0$, so we show the two-point inequality.
Pick arbitrary points $x,y \in D$. Then there exists $n\geq 1$ such that $x,y,\frac{x+y}{2}\in D_{n+1}$.
It suffices to show
\begin{equation}
\label{eq:twoppbn}
    \left| \frac{x-y}{2}\right| + \frac{B_{n+1}(x)+B_{n+1}(y)}{2} \geq B_{n+1}\left( \frac{x+y}{2} \right).
\end{equation}
Without loss of generality, assume $x\geq y$. Then there exist $k$ and $l$ such that $0\leq l \leq k \leq 2^{n+1}$,
\[
x=\frac{k}{2^{n+1}},\quad y=\frac{l}{2^{n+1}}, \quad \text{and} \quad k+l \quad \text{is even}.
\]
Rewriting \eqref{eq:twoppbn} using the definition \eqref{eq:bn} for $B_n$, we get
\[
\frac{k-l}{2^{n+2}}+\frac{(n+1)(k+l)}{2^{n+2}} - \frac{F(k)+F(l)}{2^{n+1}}\geq \frac{k+l}{2^{n+2}}(n+1) - \frac{F\left(\frac{k+l}{2}\right)}{2^n},
\]
the above line is 
equivalent to
\[
\frac{k-l}{2} +2F\left(\frac{k+l}{2}\right) \geq F(k)+F(l).
\]
Using the identity
$F(2k) = k + 2 F(k)$ from Lemma \eqref{lm:F1}, the above inequality simplifies to 
\[
F(k+l) \geq F(k)+F(l)+l,
\]
which is true by Lemma \ref{lm:F2}.

\vskip0.5cm

We are left to show that $P$ is the pointwise maximal function defined on dyadic numbers $D \subset [0,1]$, satisfying the boundary condition $P(0)=P(1)=0$ and the two-point inequality (\ref{ori}).

\vskip0.5cm

For $n \geq 1$, let $\{0,1\}^n$ be the hypercube of dimension $n$. The vertices $x, y \in\{0,1\}^n$ are joined by an edge if $x=\left(x_1, \ldots, x_n\right)$ and $y=\left(y_1, \ldots, y_n\right)$ differ by exactly one coordinate. Denote such an edge by $(x, y)$. 
Let $A\subset \{0,1\}^n$ and $A^c:=\{0,1\}^n \backslash A$, we call $y\in A^c$ a neighbor of $x\in A$ if it joins $x$ by an edge.
Define the edge boundary of $A$ by $\nabla A:=\left\{(x, y): x \in A, y \in A^c\right\}$, and the function 
\[
h_A (x) =\begin{cases}
    0&{\rm{if}}\;\; x \in A^c,\\
    \text{number of neighbors of $x$} & {\rm{if}}\;\; x\in A.
\end{cases}
\]

The quantity $\mathbb{E} h_A(x)$, where $x\sim \mathrm{unif}(\{0,1\}^{n})$,  has the property $\mathbb{E} h_A(x)=\mathbb{E} h_{A^c}(x)$, and we have $\mathbb{E} h_A(x)=\frac{|\nabla A|}{2^n}$.
It is known (see \cite{BIM}) that
\[
\min_{A\subset \{0,1\}^n,\, |A|=t} \mathbb{E} h_A(x)=P(t)
\quad \text{for all} \quad t \in D_{n}.
\]
Here $|A| = \mathbb{P}(x \in A)$ is the uniform probability measure. 
It suffices to show that 
$\E h_A(x) \geq Q(|A|)$
for any  $Q:{D}\to \mathbb{R}$ that satisfies $Q(0)=Q(1)=0$ and \eqref{ori}.

Let $f = \1_A$ be the indicator function of a set $A\subseteq \{0,1\}^n$, for $x\in \{0,1\}^n$, denote $x^i= (x_1,\cdots x_{i-1},1-x_i,x_{i+1},\cdots x_n)$ and denote $(\cdot)_{+} = \max(\cdot,0)$.
We prove the inequality 
\begin{equation}\label{eq:eha}
\E h_A 
= \E \sum_{i=1}^{n} (f(x)-f(x^i))_{+}
\geq Q(|A|)
\end{equation}
 by induction on $n$.  
For $n=1$, we have $\E h_A = \frac{|\nabla A|}{2}$. We verify that for all $A\in \{0,1\}$, \eqref{eq:eha} holds. 
Let $A = \emptyset$ (the case $A=\{0,1\}$ is similar), then $\E h_A = 0 = Q(0)$.
Let $A= \{1\}$  (the case $A=\{0\}$ is similar), then $\E h_A = \frac{1}{2}$, 
Since $Q$ satisfies \eqref{ori}, and $|A| = \frac{1}{2}$, we have 
\[
Q\left(\frac{1}{2}\right) = Q\left( \frac{1+0}{2} \right) \leq \frac{1}{2} + \frac{Q(1)+Q(0)}{2} = \frac{1}{2}= \E h_A.
\]

Assume now that the statement holds for $n-1$. For $\bar x=(x_2,\dots,x_n)\in\{0,1\}^{n-1}$, define
\[
A_i:=\{\bar x\in\{0,1\}^{n-1}:(i,\bar x)\in A\},\qquad i=0,1,
\]
and write
\[
f_i(\bar x):=\mathbbm{1}_{A_i}(\bar x),\qquad i=0,1.
\]
Here and below, $|A_i|=\mathbb P(\bar x\in A_i)$ denotes the uniform probability measure on $\{0,1\}^{n-1}$, so
\[
|A|=\frac{|A_1|+|A_0|}{2}.
\]

For $k=1,\dots,n-1$, let $\bar x^{\,k}$ denote the point obtained from $\bar x$ by flipping its $k$-th coordinate. Also define
\[
h_{A_i}(\bar x):=\sum_{k=1}^{n-1}\bigl(f_i(\bar x)-f_i(\bar x^{\,k})\bigr)_+,
\qquad i=0,1.
\]
Since
\[
h_A(x)=\sum_{j=1}^n \bigl(f(x)-f(x^j)\bigr)_+,
\]
we split the sum into the first coordinate and the remaining coordinates:
\begin{align*}
\E h_A
&=
\E_{\bar x}\E_{x_1}\bigl(f(x)-f(x^1)\bigr)_+
+
\E_{\bar x}\E_{x_1}\sum_{j=2}^n \bigl(f(x)-f(x^j)\bigr)_+.
\end{align*}
For the first term, using $(a-b)_+ + (b-a)_+ = |a-b|$, we get
\begin{align*}
\E_{\bar x}\E_{x_1}\bigl(f(x)-f(x^1)\bigr)_+
&=
\frac12 \E_{\bar x}\Bigl(\bigl(f_1(\bar x)-f_0(\bar x)\bigr)_+ + \bigl(f_0(\bar x)-f_1(\bar x)\bigr)_+\Bigr) \\
&=
\frac12 \E_{\bar x}\bigl|f_1(\bar x)-f_0(\bar x)\bigr|.
\end{align*}
For the remaining coordinates, we have
\begin{align*}
\E_{\bar x}\E_{x_1}\sum_{j=2}^n \bigl(f(x)-f(x^j)\bigr)_+
&=
\frac12 \E_{\bar x}\sum_{k=1}^{n-1}\bigl(f_1(\bar x)-f_1(\bar x^{\,k})\bigr)_+ \\
&\qquad+
\frac12 \E_{\bar x}\sum_{k=1}^{n-1}\bigl(f_0(\bar x)-f_0(\bar x^{\,k})\bigr)_+ \\
&=
\frac12 \E h_{A_1}+\frac12 \E h_{A_0}.
\end{align*}
Therefore
\[
\E h_A
=
\frac12 \E\bigl|f_1-f_0\bigr|
+\frac12 \E h_{A_1}
+\frac12 \E h_{A_0}.
\]
By Jensen's inequality,
\[
\frac12 \E\bigl|f_1-f_0\bigr|
\ge
\frac12 \bigl|\E f_1-\E f_0\bigr|
=
\frac12 \bigl||A_1|-|A_0|\bigr|.
\]
By the induction hypothesis,
\[
\E h_{A_1}\ge Q(|A_1|),\qquad \E h_{A_0}\ge Q(|A_0|).
\]
Hence
\[
\E h_A
\ge
\frac12 \bigl||A_1|-|A_0|\bigr|
+\frac12\bigl(Q(|A_1|)+Q(|A_0|)\bigr).
\]
Since $Q$ satisfies \eqref{ori},
\[
Q\!\left(\frac{|A_1|+|A_0|}{2}\right)
\le
\frac12\bigl(Q(|A_1|)+Q(|A_0|)\bigr)
+
\frac12\bigl||A_1|-|A_0|\bigr|.
\]
Therefore
\[
\E h_A\ge Q\!\left(\frac{|A_1|+|A_0|}{2}\right)=Q(|A|).
\]
This completes the induction and proves~(ii).\\

\begin{lemma}\label{lem:dyadic-symmetry}
For all $x\in D$, we have the following dyadic symmetry
\[
P(x)=P(1-x).
\]
\end{lemma}

\begin{proof}
Define
$L(x):=P(1-x)$ for $ x\in D.$   
Then
\[
L(0)=P(1)=0,\qquad L(1)=P(0)=0.
\]
Also, if $x,y\in D$, then
\begin{align*}
L\!\left(\frac{x+y}{2}\right)
&=P\!\left(1-\frac{x+y}{2}\right)
 =P\!\left(\frac{(1-x)+(1-y)}{2}\right)\\
&\le \frac{P(1-x)+P(1-y)}{2}+\frac{|(1-x)-(1-y)|}{2}\\
&= \frac{L(x)+L(y)}{2}+\frac{|x-y|}{2}.
\end{align*}
Hence $L$ satisfies the endpoint condition and \eqref{ori} on $D$.
By part~(ii), since $P$ is pointwise maximal on $D$, we get
\[
L(x)\le P(x),\qquad x\in D.
\]
Applying this with $1-x$ in place of $x$ yields
\[
P(x)=L(1-x)\le P(1-x),
\]
while applying it at $x$ gives
\[
P(1-x)=L(x)\le P(x).
\]
Therefore $
P(x)=P(1-x)$ for all $ x\in D.
$
\end{proof}

To prove (iii), we equivalently prove that for all $ x, y \in D $ such that $ |x-y| < \frac{1}{2^m} $ for some $ m \geq 1 $, we have $ |P(x) - P(y)| \lesssim \frac{m}{2^m} $.
The proof of the claim is given in the next section.\\

\begin{lemma}\label{lem:cont_ext}
The function $P$ defined on $D$, extends to a uniquely defined continuous function on $[0,1]$.
\end{lemma}

\begin{proof}
    
By part~(iii), for $x,y\in D$ with $|x-y|\le 2^{-m}$, we have
\[
|P(x)-P(y)|\le C\frac{m}{2^m}.
\]
Equivalently, if $0<|x-y|\le h\le 1/2$, then choosing $m\in\mathbb N$ such that
\[
2^{-(m+1)}<h\le 2^{-m},
\]
we obtain
\[
|P(x)-P(y)|\le C\frac{m}{2^m}
\lesssim h\log_2\!\Big(\frac1h\Big).
\]
Since $h\log_2(1/h)\to 0$ as $h\to 0^+$, it follows that $P$ is uniformly
continuous on $D$. Because $D$ is dense in $[0,1]$, $P$ extends uniquely to a
continuous function on $[0,1]$.\\
\end{proof}

To prove (iv), by Lemma \ref{lem:cont_ext}, $P$ extends uniquely to a continuous function on
$[0,1]$. By abuse of notation, we denote this extension again by $P$.

By Lemma~\ref{lem:dyadic-symmetry}, we already know that
\[
P(x)=P(1-x),\qquad x\in D.
\]
Since both sides are continuous on $[0,1]$ and agree on the dense set $D$, we obtain
\[
P(x)=P(1-x),\qquad x\in[0,1].
\]

Next, we verify that
\[
P(x)+x=2P(x/2),\qquad x\in D.
\]
Pick any $x\in D_n$. Then $x=\frac{k}{2^n}$ for some $n\ge 1$ and some integer $k$ with $0\le k\le 2^n$. Since $x\in D_n$, we have $P(x)=B_n(x)$, and since $\frac{x}{2}=\frac{k}{2^{n+1}}\in D_{n+1}$, we also have
$P(x/2)=B_{n+1}(x/2)$. By the definition of $B_n$,
\begin{align*}
P(x)+x
&=\left(\frac{nk}{2^n}-\frac{1}{2^{n-1}}F(k)\right)+\frac{k}{2^n}=\frac{(n+1)k}{2^n}-\frac{1}{2^{n-1}}F(k),
\end{align*}
while
\begin{align*}
2P\!\left(\frac{x}{2}\right)
&=2\left(\frac{(n+1)k}{2^{n+1}}-\frac{1}{2^n}F(k)\right)=\frac{(n+1)k}{2^n}-\frac{1}{2^{n-1}}F(k).
\end{align*}
Hence
\[
P(x)+x=2P(x/2).
\]
This proves (iv) on $D$.\\

To prove part (v), define
\[
g(x):=
\begin{cases}
x\log_2\!\left(\frac1x\right),& x\in(0,1],\\[1mm]
0,& x=0,
\end{cases}
\qquad
\widetilde g(x):=g(1-x),
\]
and using $\tilde{g}(X)=g(1-x)$ define
\[
B(x):=\max\{g(x),\widetilde g(x)\}
      =x^*\log_2\!\left(\frac1{x^*}\right),
\qquad x\in[0,1].
\]
Clearly,$
B(0)=B(1)=0.$
We first show that $g$ satisfies \eqref{ori} on $[0,1]$. By symmetry, the same will then
hold for $\widetilde g$.

If $x=0$ or $y=0$, the inequality is immediate. If $x=y$, it is trivial. Thus assume
\[
0<x\le y\le 1.
\]
Write
\[
y=(1+t)x,\qquad t\ge0.
\]
Then
\[
\frac{x+y}{2}=\frac{(2+t)x}{2},
\qquad
\frac{|x-y|}{2}=\frac{tx}{2}.
\]
A direct computation gives
\begin{align*}
g\!\left(\frac{x+y}{2}\right)-\frac{g(x)+g(y)}{2}
&=
\frac{x}{2}\Bigl((2+t)\Bigl(1-\log_2(2+t)\Bigr)+(1+t)\log_2(1+t)\Bigr).
\end{align*}
Therefore it is enough to prove that
\[
(2+t)\Bigl(1-\log_2(2+t)\Bigr)+(1+t)\log_2(1+t)\le t,
\qquad t\ge0.
\]
Equivalently,
\[
\phi(t):=(1+t)\log_2(1+t)-(2+t)\log_2(2+t)+2\le0.
\]
Now
\[
\phi(0)=0,
\]
and
\[
\phi'(t)=\log_2\!\left(\frac{1+t}{2+t}\right)\le0,\qquad t\ge0.
\]
Hence $\phi(t)\le0$ for all $t\ge0$, and thus $g$ satisfies \eqref{ori} on $[0,1]$.

By symmetry, $\widetilde g(x)=g(1-x)$ also satisfies \eqref{ori}. We now claim that the
maximum of two functions satisfying the endpoint condition and \eqref{ori} again satisfies
the same properties. Indeed, if $u,v$ satisfy \eqref{ori} and $w:=\max\{u,v\}$, then
\begin{align*}
w\!\left(\frac{x+y}{2}\right)
&=\max\!\left\{u\!\left(\frac{x+y}{2}\right),\,v\!\left(\frac{x+y}{2}\right)\right\}\\
&\le
\max\!\left\{
\frac{u(x)+u(y)}{2}+\frac{|x-y|}{2},
\frac{v(x)+v(y)}{2}+\frac{|x-y|}{2}
\right\}\\
&\le
\frac{\max\{u(x),v(x)\}+\max\{u(y),v(y)\}}{2}+\frac{|x-y|}{2}\\
&=
\frac{w(x)+w(y)}{2}+\frac{|x-y|}{2}.
\end{align*}
Applying this with $u=g$ and $v=\widetilde g$, we conclude that
\[
B(x)=x^*\log_2\!\left(\frac1{x^*}\right)
\]
satisfies the endpoint condition and \eqref{ori} on $[0,1]$, hence also on $D$.

By part~(ii), since $P$ is pointwise maximal on $D$, we obtain
\[
P(x)\ge B(x)=x^*\log_2\!\left(\frac1{x^*}\right),
\qquad x\in D.
\]
By continuity of $P$ from part~(iii), the same inequality holds for all $x\in[0,1]$.

It remains to prove the claimed equality cases. From part~(iv), taking $x=1$, we get
\[
P(1)+1=2P(1/2).
\]
Since $P(1)=0$, it follows that
\[
P(1/2)=\frac12.
\]
Now we prove by induction that
\[
P(2^{-k})=\frac{k}{2^k},\qquad k\ge1.
\]
This is true for $k=1$. Assume it holds for $k-1$. Applying part~(iv) with
$x=2^{-(k-1)}$, we obtain
\[
P(2^{-(k-1)})+2^{-(k-1)}=2P(2^{-k}).
\]
Hence, by the induction hypothesis,
\[
2P(2^{-k})
=
\frac{k-1}{2^{k-1}}+\frac{1}{2^{k-1}}
=
\frac{k}{2^{k-1}},
\]
so
\[
P(2^{-k})=\frac{k}{2^k}.
\]
By the symmetry proved in part~(iv),
\[
P(1-2^{-k})=P(2^{-k})=\frac{k}{2^k}.
\]
Finally,
\[
2^{-k}\log_2(2^k)=\frac{k}{2^k},
\]
so equality holds at
\[
x=2^{-k}
\qquad\text{and}\qquad
x=1-2^{-k}.
\]
This proves~(v).\\

To prove part  (vi), by Lemma \ref{lem:cont_ext}, the function $P:D\to[0,\infty)$ extends uniquely to a continuous
function on $[0,1]$. By abuse of notation, we denote this extension again by $P$.
Since $P(0)=P(1)=0$ on $D$, continuity gives
\[
P(0)=P(1)=0
\]
on $[0,1]$ as well.

We claim that this continuous extension still satisfies \eqref{ori} on all of $[0,1]$.
Fix $x,y\in[0,1]$. For each $n\ge1$, let
\[
x_n:=\frac{\lfloor 2^n x\rfloor}{2^n},
\qquad
y_n:=\frac{\lfloor 2^n y\rfloor}{2^n}.
\]
Then $x_n,y_n\in D$, and
\[
x_n\to x,\qquad y_n\to y,\qquad \frac{x_n+y_n}{2}\to \frac{x+y}{2}.
\]
Since $P$ satisfies \eqref{ori} on $D$, we have
\[
P\!\left(\frac{x_n+y_n}{2}\right)
\le
\frac{P(x_n)+P(y_n)}{2}
+\frac{|x_n-y_n|}{2}.
\]
Letting $n\to\infty$ and using continuity of $P$, we obtain
\[
P\!\left(\frac{x+y}{2}\right)
\le
\frac{P(x)+P(y)}{2}
+\frac{|x-y|}{2},
\qquad x,y\in[0,1].
\]
Thus the continuous extension of $P$ satisfies the endpoint condition and
\eqref{ori} on $[0,1]$.
By the definition of $B_{1,1}$ as the pointwise maximal continuous function on
$[0,1]$ satisfying the endpoint condition and \eqref{ori}, it follows that
\[
P(x)\le B_{1,1}(x),\qquad x\in[0,1].
\]
In particular,
\[
P(x)\le B_{1,1}(x),\qquad x\in D.
\]

On the other hand, the restriction $B_{1,1}|_D$ satisfies
\[
B_{1,1}(0)=B_{1,1}(1)=0
\]
and \eqref{ori} on $D$, since $B_{1,1}$ satisfies these properties on all of
$[0,1]$. Therefore, by part~(ii), the dyadic maximality of $P$ gives
\[
B_{1,1}(x)\le P(x),\qquad x\in D.
\]
Hence
\[
B_{1,1}(x)=P(x),\qquad x\in D.
\]

Since both $P$ and $B_{1,1}$ are continuous on $[0,1]$ and agree on the dense set
$D$, they in fact agree on all of $[0,1]$. In particular,
\[
B_{1,1}|_D=P.
\]
This proves~(vi).

\section{Proof of  Theorem~\ref{thm-bn}.iii }
In this section, we verify part (iii) of Theorem~\ref{thm-bn}. It suffices to show that there exists a universal constant $C>0$ such that for any $m\geq 1$ if $|x-y|< 2^{-m}$ where $x,y \in D$, then $|P(x)-P(y)|\leq Cm/2^{m}$.

For any $x,y\in D$ with $0<x,y<1$, there exists $n\ge 1$ such that
$x,y\in D_n$, $x=\frac{k}{2^n}$, and $y=\frac{l}{2^n}$ for some
$0<k,l<2^n$.
The condition $|x-y|<2^{-m} $ implies
$|k-l|\leq 2^{n-m}$.
Using the definition of $B_n$ we have 
\begin{align*}
|B_n(x) - B_n(y)| =& \left| B_n\left(\frac{k}{2^n}\right) -B_n\left(\frac{l}{2^n}\right)
    \right|\\
    =&
    \left|
\frac{ (n+1) (k-l) - (F(2k) - F(2l))  }{2^n}
    \right|.
\end{align*}
For such $k,l$, it suffices to show
\begin{equation}
     \left|\frac{ n (k-l) - (F(2k) - F(2l))  }{2^n}
    \right|\leq C\frac{m}{2^{m}}\label{eq:whatweneed}.
\end{equation}

Recall $F(0)=0$, and  $F(k)=\sum_{i=0}^{k-1} s(i)$, where $s(i)$ is the number of $1$'s in the binary representation of $i$. 
From {\cite[Remark 2.5]{hart}} we have
    \begin{equation*}
        F(l) = \max_{0\leq m\leq l/2} (F(m)+F(l-m)+m).
    \end{equation*}
For any integer $0 < k \leq 2^n$, its binary representation can be written as
\[
k = 2^{k_1} + 2^{k_2} + \cdots + 2^{k_T},
\]
where 
\[
n \geq k_1 > k_2 > \cdots > k_T \geq 0.
\]

We further make the following observations.

\begin{lemma}\label{lm:obv1}
    For any integer $k>0$ and integer $p$ such that $0 \leq p \leq 2^k$, it holds 
    $$F(2^k + p) = k 2^{k-1} + p + F(p).$$
\end{lemma}
\begin{proof}
    It is easy to see that 
    \begin{equation}
    \label{eq:F2k}
           F(2^k) = \sum_{j=0}^{2^k-1}s(j)=k2^{k-1},
    \end{equation}
    and for all $0\leq j< 2^{k}$ we have
    \[
    s(2^{k}+j) = 1+s(j) .
    \]
Let $0 \leq p \leq  2^k$.    We have
    \begin{align*}
        F(2^k + p) &= \sum_{j=0}^{2^k + p - 1} s(j) 
        = \sum_{j=0}^{2^k - 1} s(j) + \sum_{j=2^k}^{2^k + p - 1} s(j) \\
        &= F(2^k) + \sum_{j=0}^{p-1} s(2^k + j) = k 2^{k-1} + \sum_{j=0}^{p-1} \left(1 + s(j)\right) \\
        &= k 2^{k-1} + p + F(p).
    \end{align*}
    This finishes the proof of the lemma. 
\end{proof}

\begin{lemma}\label{lem:F_of_bin_rep}
    For any $0< k\leq 2^n$ of the form 
    $$k = 2^{k_1} + 2^{k_2} + \cdots + 2^{k_T},$$ where $n\geq k_1 > k_2 > \cdots \geq 0$ it holds 
    $$F(k) = \sum_{j=1}^{T} (k_j+2(j-1))2^{k_j-1}.$$
\end{lemma}
\begin{proof}
    
    Let $p= k-2^{k_1}$, then $p<2^{k_1}$.
    Hence by Lemma \ref{lm:obv1} we have
    \begin{equation}\label{iteration1}
        F(k) = F(2^{k_1}+(k-2^{k_1})) = k_1 2^{k_1-1}+ (k-2^{k_1}) + F(k-2^{k_1}).
    \end{equation}
    Take $p=k-2^{k_1}-2^{k_2}$, then $p<2^{k_2}$ Using Lemma \ref{lm:obv1} again, we have
    \begin{equation}
    \label{iteration2}
\begin{aligned}
    F(k-2^{k_1})  
    &= F(2^{k_2}+(k - 2^{k_{1}} - 2^{k_2})) \\
    &= k_22^{k_2-1} + (k-2^{k_1}-2^{k_2})  + F(k-2^{k_1}-2^{k_2}).
\end{aligned}
    \end{equation}
    Combining \eqref{iteration1} and \eqref{iteration2}, we obtain
    \begin{align*}
        F(k)&=k_1 2^{k_1-1}+ (k-2^{k_1}) +k_22^{k_2-1} + (k-2^{k_1}-2^{k_2})  + F(k-2^{k_1}-2^{k_2})\\
        &=k_1 2^{k_1-1}+ (2^{k_2} + \cdots + 2^{k_T}) +k_22^{k_2-1} + (2^{k_3} + \cdots + 2^{k_T}) \\
        &\;\;+ F(k-2^{k_1}-2^{k_2}).
    \end{align*}
    Iterating this process, we obtain 
    \begin{align*}
        F(k)
        &=k_1 2^{k_1-1}+ (2^{k_2} + \cdots + 2^{k_T}) +k_22^{k_2-1} + (2^{k_3} + \cdots + 2^{k_T}) +\\
        &\quad\cdots  + k_{T-1}2^{k_{T-1}-1}+2^{k_T} + F(2^{k_T})\\
        &= \sum_{j=1}^{T} k_j 2^{k_j-1} + \sum_{j=2}^{T} (j-1)2^{k_j}\\ 
    \end{align*}
\end{proof}

\begin{lemma}\label{claim0}
Let $0 \leq l < k \leq 2^n$ such that $|k - l| \leq 2^{n-m}$; consider the binary representations
\begin{align*}
    l=2^{l_1} + 2^{l_2} +\cdots +2^{l_R},\qquad 
    k=2^{k_1} + 2^{k_2} +\cdots +2^{k_T},
\end{align*}
where $k_1 > k_2 > \cdots \geq 0$ and $l_1>l_2>\cdots >l_R\geq 0$.
Then, for a fixed $p>0$, if $l_i = k_i$ for all $i = 1, \dots, p-1$ and $k_p \geq l_p + 2$, we have 
\[
k_p \leq n - m + 1.
\]
\end{lemma}
\begin{proof}
Fix $p>0$, such that $l_i=k_i$ for all $i=1, \ldots, p-1$, and $k_p\geq l_p+2$.
Using the finite geometric series formula, we have
\begin{equation}\label{eq:1234}
    2^{k_p-1}-1 = 2^{k_p-2} +2^{k_p-3}+\cdots 2^{0}\geq 2^{l_p}+2^{l_p-1}+\cdots+2^{0}.
\end{equation}
From \eqref{eq:1234} we have 
\begin{align*}
    |k-l|&=2^{k_p}+\cdots+ 2^{k_T} - (2^{l_p} +\cdots +2^{l_R})\\
    &\geq 2^{k_p} - (2^{k_p-1}-1)= 2^{k_p-1}+1.
\end{align*}
 Since $|k-l|\leq 2^{n-m}$, it follows  $2^{k_p-1}+1 \leq 2^{n-m}$. Hence,
 \[
 k_p-1\leq n-m.
 \]
 This completes the proof of the lemma.
\end{proof}

\begin{lemma}\label{obs2}
Let $0 \leq l < k \leq 2^n$, such that $|k - l| \leq 2^{n-m}$, and having the binary representations
\begin{align*}
    l&=2^{l_1} + 2^{l_2} +\cdots +2^{l_R},\qquad 
    k&=2^{k_1} + 2^{k_2} +\cdots +2^{k_T}, 
\end{align*}
where
\[
n\geq k_1 > k_2 > \cdots \geq 0,
\qquad
n\geq l_1>l_2>\cdots >l_R\geq 0.
\]
Fix $p>0$ and assume
\[
l_i = k_i \quad\text{for } i=1,\dots,p-1,
\qquad
k_p = l_p + 1.
\]
Let $M\ge 1$ be maximal such that
\[
l_{p+j}=k_p-(j+1),\qquad j=0,\dots,M-1.
\]
Equivalently, either $p+M>R$, or else
\[
l_{p+M}\le k_p-(M+2).
\]
Then
\begin{equation}\label{eq1}
p<T \Longrightarrow k_{p+1}\le n-m+1.
\end{equation}
Moreover,
\begin{equation}\label{eq2}
k_p-M\le n-m+1.
\end{equation}
If $p+M\le R$, then
\begin{equation}\label{eq3}
l_{p+M}\le n-m+1.
\end{equation}
\end{lemma}

\begin{proof}
Assume first that $p<T$, so that $k_{p+1}$ is defined. Since
$l_i=k_i$ for $i=1,\dots,p-1$ and $k_p=l_p+1$, we have
\begin{align*}
|k-l|
&=2^{k_p}+2^{k_{p+1}}+\cdots+2^{k_T}-(2^{l_p}+\cdots+2^{l_R})\\
&\ge 2^{k_p}+2^{k_{p+1}}+\cdots+2^{k_T}-(2^{k_p-1}+\cdots+2^0)\\
&=2^{k_p}+2^{k_{p+1}}+\cdots+2^{k_T}-(2^{k_p}-1)\\
&\ge 2^{k_{p+1}}.
\end{align*}
Since $|k-l|\le 2^{n-m}$, it follows that $k_{p+1}\le n-m$, and hence
\eqref{eq1} holds.

Next we prove \eqref{eq2}. By maximality of $M$, the block
\[
2^{l_p}+2^{l_{p+1}}+\cdots+2^{l_{p+M-1}}
\]
is exactly
\[
2^{k_p-1}+2^{k_p-2}+\cdots+2^{k_p-M}.
\]
Therefore
\begin{align*}
|k-l|
&=2^{k_p}+\cdots+2^{k_T}
-\bigl(2^{k_p-1}+\cdots+2^{k_p-M}\bigr)
-\sum_{j=p+M}^{R}2^{l_j}.
\end{align*}
If $p+M>R$, then the last sum is empty, so
\[
|k-l|\ge 2^{k_p-M}.
\]
If $p+M\le R$, then by maximality of $M$ we have
$l_{p+M}\le k_p-M-2$, and hence
\[
\sum_{j=p+M}^{R}2^{l_j}\le 2^{k_p-M-2}+2^{k_p-M-3}+\cdots+2^0
=2^{k_p-M-1}-1.
\]
Thus in either case,
\[
|k-l|\ge 2^{k_p-M-1}.
\]
Since $|k-l|\le 2^{n-m}$, we obtain
\[
k_p-M\le n-m+1,
\]
which proves \eqref{eq2}.

Finally, if $p+M\le R$, then by maximality of $M$,
\[
l_{p+M}\le k_p-M-2.
\]
Combining this with \eqref{eq2} gives
\[
l_{p+M}\le n-m-1\le n-m+1,
\]
which proves \eqref{eq3}.
\end{proof}

\begin{lemma}{\label{lemma_summation}}
    For $M\geq 1$, we have
    \[
    \sum_{j=1}^{M}j2^{-j} = 2-2^{-M}(M+2).
    \]
\end{lemma}
\begin{proof}
   Indeed, recall that
    \[
    \sum_{j=0}^{M}x^j = \frac{1-x^{M+1}}{1-x}.
    \]
    Differentiating both sides and then multiplying by $x$, yields 
    \[
\sum_{j=1}^{M} jx^{j}
= x\frac{\Big(-(M+1)x^M(1-x) + (1-x^{M+1})\Big)}{(1-x)^2}.
\]
    Substituting $x=2^{-1}$, the identity follows.
\end{proof}

We first dispose of the endpoint cases. If one of $x,y$ equals $1$, then by
Lemma~\ref{lem:dyadic-symmetry},
\[
|P(x)-P(y)|=|P(1-x)-P(1-y)|,
\]
and one of $1-x,1-y$ equals $0$. Hence it is enough to treat the case where one
of the two points is $0$.
By symmetry in $x$ and $y$, we may assume
\[
y=0,\qquad x=\frac{k}{2^n}\in D_n,\qquad 1\le k\le 2^{n-m},
\]
so that $|x-y|=x\le 2^{-m}$. Write
\[
k=2^{k_1}+\cdots+2^{k_T},\qquad n\ge k_1>k_2>\cdots>k_T\ge 0.
\]
Since $k\le 2^{n-m}$, we have $k_1\le n-m$, and therefore
\[
n-k_j\ge m+j-1\ge m,\qquad j=1,\dots,T.
\]

Using Lemma~\ref{lem:F_of_bin_rep} on $F(2k)$, we get
\[
\frac{|nk-F(2k)|}{2^n}
=
\frac1{2^n}\left|
\sum_{j=1}^{T}(n-k_j-2j+1)2^{k_j}
\right|.
\]
Hence
\[
\frac{|nk-F(2k)|}{2^n}
\le
\sum_{j=1}^{T}\frac{n-k_j}{2^{\,n-k_j}}
+
2\sum_{j=1}^{T}\frac{j}{2^{\,n-k_j}}
+
\frac{k}{2^n}.
\]

For the first sum,
\[
\sum_{j=1}^{T}\frac{n-k_j}{2^{\,n-k_j}}
\le
\sum_{r=m}^{\infty}\frac{r}{2^r}
\lesssim \frac{m}{2^m}.
\]
For the second sum,
\[
\sum_{j=1}^{T}\frac{j}{2^{\,n-k_j}}
\le
\sum_{j=1}^{\infty}\frac{j}{2^{m+j-1}}
\lesssim 2^{-m}
\lesssim \frac{m}{2^m}.
\]
Finally,
\[
\frac{k}{2^n}\le 2^{-m}\lesssim \frac{m}{2^m}.
\]
Thus the required bound holds whenever one of $x,y$ belongs to $\{0,1\}$.

It remains to consider the case $0<x,y<1$.

For $x,y\in D$ with $0<x,y<1$ and $|x-y|\le 2^{-m}$, write
$x=\frac{k}{2^n}$ and $y=\frac{l}{2^n}$ with $0<k,l<2^n$. To prove
$|P(x)-P(y)|\le C m/2^m$, it is enough to establish \eqref{eq:whatweneed}, namely
\begin{equation}\label{gant01}
     \left|\frac{ n (k-l) - (F(2k) - F(2l))  }{2^n}
    \right|\leq C\frac{m}{2^{m}}.
\end{equation}
Consider the following binary representation of $k,l$ and $2k,2l$,
\begin{equation}\label{eq:binaryrep}
    \begin{aligned}
    l&=2^{l_1} + 2^{l_2} +\cdots +2^{l_R},&
    k&=2^{k_1} + \cdots +2^{k_T},\\
    2l&=2^{l_1+1} + 2^{l_2+1} +\cdots +2^{l_R+1},&
    2k&=2^{k_1+1} +\cdots +2^{k_T+1}.
\end{aligned}
\end{equation}
Since $|k-l|\leq  2^{n-m}$, we know 
\begin{equation}\label{eq:k-l}
    |2^{l_1} + 2^{l_2} +\cdots +2^{l_R} - 2^{k_1} - 2^{k_2} -\cdots -2^{k_T}| \leq 2^{n-m}.
\end{equation}

Using Lemma \ref{lem:F_of_bin_rep} on $F(2l)$ and $F(2k)$, 
we rewrite \eqref{gant01} as follows
\begin{align*}
    &\frac{1}{{2^n}}
    \left| 
    \sum_{j=1}^{T}n2^{k_j} - \sum_{j=1}^{R}n2^{l_j}
     \right.\\
     &\left. -\left(
    \sum_{j=1}^{T}((k_j+1)+2(j-1))2^{k_j}
    -\sum_{j=1}^{R}((l_j+1)+2(j-1))2^{l_j}
    \right)
    \right|
    \leq C\frac{m}{2^{m}},
\end{align*}
which simplifies to
\begin{align*}
    \frac{1}{{2^n}}
    \Bigg|
    \sum_{j=1}^{T}(n - k_j- 2j +1 )2^{k_j}
    -\sum_{j=1}^{R}( n - l_j- 2j+1 )2^{l_j}
    \Bigg|
    \leq C\frac{m}{2^{m}}.
\end{align*}
Using the triangle inequality, it is sufficient to show
\begin{align*}
    \frac{1}{{2^n}}
    \Bigg|
    \sum_{j=1}^{T}(n - k_j- 2j  )2^{k_j}
    -\sum_{j=1}^{R}( n - l_j- 2j )2^{l_j}
    \Bigg|
    + \left|\frac{k-l}{2^n}\right|
    \leq C\frac{m}{2^{m}}.
\end{align*}
Since $ \left|\frac{k-l}{2^n}\right|<2^{-m}$, it is enough to show 
\begin{equation}\label{eq:goal_bound}
        \frac{1}{{2^n}}
    \Bigg|
    \sum_{j=1}^{T}(n - k_j- 2j  )2^{k_j}
    -\sum_{j=1}^{R}( n - l_j- 2j )2^{l_j}
    \Bigg|
    \leq C\frac{m}{2^{m}}.
\end{equation}
Without loss of generality, assume $l<k$ (the case $k=l$ is trivial).
We first dispose of the case in which the binary expansion of $l$ is an initial segment of that of $k$, namely
\[
R<T,
\qquad
l_i=k_i \quad \text{for } i=1,\dots,R.
\]
Then
\begin{align*}
\frac{1}{{2^n}}
\Bigg|
\sum_{j=1}^{T}(n - k_j- 2j )2^{k_j}
-\sum_{j=1}^{R}( n - l_j- 2j )2^{l_j}
\Bigg|
&=
\frac{1}{{2^n}}
\Bigg|
\sum_{j=R+1}^{T}(n - k_j- 2j )2^{k_j}
\Bigg|\\
&\le
\sum_{j=R+1}^{T}\frac{|n-k_j|}{2^{\,n-k_j}}
+
2\sum_{j=R+1}^{T}\frac{j}{2^{\,n-k_j}}.
\end{align*}
Since
\[
k-l=\sum_{j=R+1}^{T}2^{k_j}\le 2^{n-m},
\]
we have $k_{R+1}\le n-m$, and therefore
\[
n-k_{R+1}\ge m.
\]
Moreover, since $n\ge k_1$ and $k_1>k_2>\cdots>k_T$, we have
\[
n-(j-1)\ge k_j,
\qquad j=1,\dots,T,
\]
hence
\[
j\le n-k_j+1,
\qquad j=1,\dots,T.
\]
Because the integers
\[
n-k_{R+1},\dots,n-k_T
\]
are distinct and all belong to $\{m,m+1,\dots\}$, it follows that
\begin{align*}
\sum_{j=R+1}^{T}\frac{|n-k_j|}{2^{\,n-k_j}}
&\le \sum_{r=m}^{\infty}\frac{r}{2^r}
\lesssim \frac{m}{2^m},\\
\sum_{j=R+1}^{T}\frac{j}{2^{\,n-k_j}}
&\le \sum_{j=R+1}^{T}\frac{n-k_j+1}{2^{\,n-k_j}}
\le \sum_{r=m}^{\infty}\frac{r+1}{2^r}
\lesssim \frac{m}{2^m}.
\end{align*}
Thus \eqref{eq:goal_bound} holds in this initial-segment case.

We may therefore assume that there exists an integer $p\le \min\{R,T\}$ such that
\[
k_i=l_i \quad \text{for } i=1,\dots,p-1,
\qquad
k_p>l_p.
\]
Let $p$ be the first such index. Then
\begin{align*}
    &\frac{1}{{2^n}}
    \Bigg|
    \sum_{j=1}^{T}(n - k_j- 2j  )2^{k_j}
    -\sum_{j=1}^{R}( n - l_j- 2j )2^{l_j}
    \Bigg|\\
    &=\frac{1}{{2^n}}
    \Bigg|
    \sum_{j=p}^{T}(n - k_j- 2j  )2^{k_j}
    -\sum_{j=p}^{R}( n - l_j- 2j )2^{l_j}
    \Bigg|.
\end{align*}
To show \eqref{eq:goal_bound}, we consider two cases, $k_p= l_p+1$ and $k_p\geq l_p+2$.\\

\textbf{Case 1} (simple case). Consider the case
$k_p\geq l_p+2$. To verify \eqref{eq:goal_bound}, observe that from the triangle inequality, it follows
\begin{equation}\label{eq:summation}
    \begin{aligned}
    &\frac{1}{{2^n}}
    \Bigg|
    \sum_{j=p}^{T}(n - k_j- 2j  )2^{k_j}
    -\sum_{j=p}^{R}( n - l_j- 2j )2^{l_j}
    \Bigg|\\
    &\leq 
    \sum_{j=p}^{T}\frac{|n - k_j- 2j  |}{2^{n-k_j}}
    +\sum_{j=p}^{R}\frac{| n - l_j- 2j |}{2^{n-l_j}}\\
    &\leq 
    \underbrace{\sum_{j=p}^{T}\frac{|n - k_j|}{2^{n-k_j}}}_{(I)}
    +2\underbrace{\sum_{j=p}^{T}\frac{j}{2^{n-k_j}}}_{(II)}
    +\underbrace{\sum_{j=p}^{R}\frac{| n - l_j |}{2^{n-l_j}}}_{(III)}
    +2\underbrace{\sum_{j=p}^{R}\frac{j}{2^{n-l_j}}}_{(IV)}.
\end{aligned}
\end{equation}
We study each summation separately. 
For summation $(I)$, by Lemma~\ref{claim0} we have $n-k_p\ge m-1$. Since
\[
k_p>k_{p+1}>\cdots >k_T,
\]
the integers
\[
n-k_p,\; n-k_{p+1},\;\dots,\; n-k_T
\]
are distinct and all belong to $\{m-1,m,m+1,\dots\}$. Therefore
\begin{align*}
(I)
= \sum_{j=p}^{T}\frac{|n-k_j|}{2^{n-k_j}}
\le \sum_{r=m-1}^{\infty}\frac{r}{2^r}
\lesssim \frac{m}{2^m}.
\end{align*}

To bound $(III)$, notice that $l_p+2\le k_p\le n-m+1$, hence
$n-l_p\ge m+1$. Since
\[
l_p>l_{p+1}>\cdots >l_R,
\]
the integers
\[
n-l_p,\; n-l_{p+1},\;\dots,\; n-l_R
\]
are distinct and all belong to $\{m+1,m+2,\dots\}$. Therefore
\begin{align*}
(III)
= \sum_{j=p}^{R}\frac{|n-l_j|}{2^{n-l_j}}
\le \sum_{r=m+1}^{\infty}\frac{r}{2^r}
\lesssim \frac{m}{2^m}.
\end{align*}

To estimate $(II)$, notice that by binary representation of $k$ (see \eqref{eq:binaryrep}) we have $n\geq k_1$. Hence it follows  $n-1\geq k_2$, and by iterating we get  $n-(j-1)\geq k_j$, i.e., 
\begin{align}\label{iter01}
    j \leq n-k_{j}+1
\end{align}

Thus
\begin{align*}
(II)
=\sum_{j=p}^{T}\frac{j}{2^{n-k_j}}\le \sum_{j=p}^{T}\frac{n-k_j+1}{2^{n-k_j}} \le \sum_{r=m-1}^{\infty}\frac{r+1}{2^r}
\lesssim \frac{m}{2^m}.
\end{align*}

To bound $(IV)$, we first note that, exactly as for the sequence $(k_j)$, the
strict inequalities
$
n\ge l_1>l_2>\cdots>l_R\ge 0
$ 
imply
\[
n-(j-1)\ge l_j,\qquad j=1,\dots,R,
\]
and hence
\begin{equation}\label{iter01-l}
j\le n-l_j+1,\qquad j=1,\dots,R.
\end{equation}
Since $n-l_p\ge m+1$, the numbers
\[
n-l_p,\;n-l_{p+1},\;\dots,\;n-l_R
\]
are distinct and all belong to $\{m+1,m+2,\dots\}$. Therefore
\begin{align*}
(IV)
&=\sum_{j=p}^{R}\frac{j}{2^{n-l_j}}
\le \sum_{j=p}^{R}\frac{n-l_j+1}{2^{n-l_j}} \le \sum_{r=m+1}^{\infty}\frac{r+1}{2^r}
\lesssim \frac{m}{2^m}.
\end{align*}

This completes the proof in the case $k_p\geq l_p+2$.\\

\textbf{Case 2.} For the same $l,k$ given by \eqref{eq:binaryrep}, assume
\[
l_i=k_i \quad\text{for } i=1,\dots,p-1,
\qquad
l_p=k_p-1.
\]
Let $M\ge 1$ be maximal such that
\[
l_{p+j}=k_p-(j+1),\qquad j=0,\dots,M-1.
\]
Thus either $p+M>R$, or else
\[
l_{p+M}\le k_p-(M+2).
\]
We have
\begin{align*}
     &\frac{1}{{2^n}}
    \Big|
    \sum_{j=1}^{T}(n - k_j- 2j  )2^{k_j}
    -\sum_{j=1}^{R}( n - l_j- 2j )2^{l_j}
    \Big|\\
    &=\frac{1}{{2^n}}
    \Big|
    \sum_{j=p}^{T}(n - k_j- 2j  )2^{k_j}
    -\sum_{j=p}^{R}( n - l_j- 2j )2^{l_j}
    \Big|\\
    &\leq 
    \underbrace{
    \frac{1}{2^n}\Big|
    (n-k_p-2p){2^{k_p}} - \sum_{j=p}^{p+M-1} (n-l_j -2j) 2^{l_j}
    \Big|}_{(I)}\\
    &\qquad+
    \underbrace{
    \frac{1}{2^n}\sum_{j=p+1}^{T} |n-k_j -2j|2^{k_j} 
    }_{(II)}
    +
    \underbrace{
    \frac{1}{2^n} \sum_{j=p+M}^{R} |n-l_j-2j|2^{l_j}
    }_{(III)}.
\end{align*}
Here the sum in $(II)$ is understood to be $0$ when $p=T$, and the sum in
$(III)$ is understood to be $0$ when $p+M>R$.

We bound each summation individually. For summation $(II)$, if $p=T$, then
$(II)=0$. Otherwise, by the triangle inequality we have
\begin{align*}
    (II)\leq 
    \underbrace{\sum_{j=p+1}^{T} \frac{|n-k_j|}{2^{n-k_j}}}_{(II.1)}
    +2\underbrace{\sum_{j=p+1}^{T} \frac{j}{2^{n-k_j}}}_{(II.2)}.
\end{align*}
Since $p<T$, Lemma~\ref{obs2}~\eqref{eq1} gives $n-k_{p+1}\ge m-1$. Thus
\[
\{n-k_{p+1},\dots,n-k_T\}\subseteq \{m-1,m,m+1,\dots\},
\]
and hence
\begin{align*}
    (II.1)
    &\le \sum_{j=m-1}^{\infty}\frac{j}{2^j}
    \lesssim \frac{m}{2^m},\\
    (II.2)
    &\le \sum_{j=p+1}^{T}\frac{n-k_j+1}{2^{n-k_j}}
    \le \sum_{j=m-1}^{\infty}\frac{j}{2^j}
      +\sum_{j=m-1}^{\infty}\frac{1}{2^j}
    \lesssim \frac{m}{2^m}.
\end{align*}
Therefore
\[
(II)\lesssim \frac{m}{2^m}.
\]

To bound summation $(III)$, if $p+M>R$, then $(III)=0$. Otherwise,
\begin{align*}
(III)
&=\sum_{j=p+M}^{R} \frac{|n-l_j - 2j|2^{l_j}}{2^n} \\
&\le \sum_{j=p+M}^{R} \frac{n-l_j}{2^{n-l_j}}
+ 2 \sum_{j=p+M}^{R} \frac{j}{2^{n-l_j}}.
\end{align*}
By Lemma~\ref{obs2}~\eqref{eq3}, we have $n-l_{p+M}\ge m-1$. Also, since
\[
n\ge l_1>l_2>\cdots>l_R\ge 0,
\]
we have
\[
n-(j-1)\ge l_j,\qquad j=1,\dots,R,
\]
hence
\[
j\le n-l_j+1,\qquad j=1,\dots,R.
\]
Therefore the numbers
\[
n-l_{p+M},\dots,n-l_R
\]
are distinct and all belong to $\{m-1,m,m+1,\dots\}$, so
\begin{align*}
\sum_{j=p+M}^{R} \frac{n-l_j}{2^{n-l_j}}
&\le \sum_{r=m-1}^{\infty}\frac{r}{2^r}
\lesssim \frac{m}{2^m},\\
\sum_{j=p+M}^{R} \frac{j}{2^{n-l_j}}
&\le \sum_{j=p+M}^{R} \frac{n-l_j+1}{2^{n-l_j}}
\le \sum_{r=m-1}^{\infty}\frac{r+1}{2^r}
\lesssim \frac{m}{2^m}.
\end{align*}
Thus
\[
(III)\lesssim \frac{m}{2^m}.
\]

Finally, to bound summation $(I)$, recall that $l_{p+j} = k_p - (j+1)$ for $j=0,\ldots, M-1$, and, therefore, 
$$n-l_{p+j} -2(p+j)=n-k_p +(j+1) -2(p+j) =n-k_p - 2p-j+1$$
for all $j=0, \ldots, M-1$. 
Hence
\begin{equation}\label{eq:sumI}
    \begin{aligned}
    (I)= &\frac{1}{2^n}\Big|
    (n-k_p-2p){2^{k_p}} - \sum_{j=p}^{p+M-1} (n-l_j -2j) 2^{l_j}
    \Big| \\
    =& \frac{1}{2^n} \Big|
    (n-k_p-2p){2^{k_p}} -  \sum_{j=0}^{M-1} (n-k_p - 2p-j+1) 2^{ k_p-(j+1) }
    \Big|\\
    =& \frac{1}{2^n} \Big|
    (n-k_p-2p){2^{k_p}} - \underbrace{ \sum_{i=1}^{M} (n-k_p - 2p-i+2) 2^{ k_p-i }}_{(I.1)}
    \Big|.\\
\end{aligned}
\end{equation}
We now study $(I.1)$. Recall from Lemma \ref{lemma_summation} that
\[ \sum_{j=1}^{M} 2^{k_p-j} = 2^{k_p} - 2^{k_p-M} ,\quad \sum_{j=1}^{M}j 2^{k_p-j}=2^{k_p}(2-2^{-M}(M+2)).\]
So, 
\begin{align*}
    (I.1)
    &=\sum_{j=1}^{M} (n-k_p - 2p)2^{k_p-j}+2\sum_{j=1}^{M} 2^{k_p-j} - \sum_{j=1}^{M} j2^{k_p-j} \\
    &= (n-k_p -2p)(2^{k_p}-2^{k_p-M})  +2 (2^{k_p}-2^{k_p-M})-2^{k_p}(2-2^{-M}(M+2))\\
    &= (n-k_p-2p) 2^{k_p} - 2^{k_p-M} (n-k_p- 2p-M).
\end{align*}
Plugging the above into \eqref{eq:sumI} yields,
\begin{align*}
    (I)=& \frac{1}{2^n}\left|(n-k_p-2p){2^{k_p}} -(n-k_p-2p) 2^{k_p} + 2^{k_p-M} (n-k_p- 2p-M) \right|\\
    =&\frac{1}{{2^{n-k_p+M}}}\left|{n-k_p-M-2p}\right|
    \leq  \frac{|n-k_p+M|}{2^{n-k_p+M}} + \frac{2(p+M)}{2^{n-k_p+M}}.
\end{align*}

Set
\[
t:=n-k_p+M.
\]
By Lemma~\ref{obs2} \eqref{eq2}, we have $t\ge m-1$. Also, by \eqref{iter01},
\[
p\le n-k_p+1,
\]
and hence
\[
p+M\le n-k_p+M+1=t+1.
\]
Therefore
\begin{align*}
(I)
&\le \frac{t}{2^t}+\frac{2(p+M)}{2^t}\\
&\le \frac{t}{2^t}+\frac{2(t+1)}{2^t}
\le \frac{3t+2}{2^t}
\le 3\sum_{r=t}^{\infty}\frac{r+1}{2^r}
\le 3\sum_{r=m-1}^{\infty}\frac{r+1}{2^r}
\lesssim \frac{m}{2^m}.
\end{align*}
This completes the proof in the case $k_p=l_p+1$,
and so the proof of part~(iii) of Theorem~\ref{thm-bn} is complete. In particular, the proof of Theorem~\ref{thm-bn} is complete.\\

\end{document}